\documentclass[UTF8]{article}
\usepackage{blindtext}
\usepackage[a4paper, total={6in, 8in}]{geometry}

\usepackage{cite,indentfirst,amsmath,amssymb,mathrsfs,graphicx,subfig,epstopdf,booktabs,multirow}
\usepackage{verbatim,enumerate,enumitem}
\usepackage{authblk}

\newtheorem{theorem}{Theorem}[section]

\newenvironment{proof}{{\noindent\it Proof:}}{\hfill $\square$\par}

\title{The high-order exponential semi-implicit scalar auxiliary variable approach for nonlocal Cahn-Hilliard equation}
\author[a]{Xiaoqing Meng}
\author[a]{Aijie Cheng \thanks{Corresponding author: aijie@sdu.edu.cn}}
\author[b]{Zhengguang Liu}
\affil[a]{School of Mathematics, Shandong University, Jinan, Shandong 250100, China.}
\affil[b]{School of Mathematics and Statistics, Shandong Normal University, Jinan, Shandong 250014, China.}

\date{}

\begin{document}
\maketitle 

\begin{abstract}
The nonlocal Cahn-Hilliard (NCH) equation with nonlocal diffusion operator
is more suitable for the simulation of microstructure phase transition 
than the local Cahn-Hilliard (LCH) equation. 
In this paper, based on the
exponential semi-implicit scalar auxiliary variable (ESI-SAV) 
method, the highly effcient and accurate schemes in time with 
unconditional energy stability for solving 
the NCH equation are proposed. 
On the one hand, we have demostrated the unconditional energy stability for the NCH equation
with its high-order semi-discrete schemes carefully and rigorously. 
On the other hand, in order to reduce the calculation and storage cost in 
numerical simulation, we use the fast solver based on FFT and FCG for 
spatial discretization. 
Some numerical simulations involving the Gaussian 
kernel are presented and show the stability, accuracy, efficiency 
and unconditional energy stability of the proposed schemes.
\\
\noindent{\textbf{keywords}}
\\
High-order schemes; Exponential semi-implicit scalar auxiliary variable;
Nonlocal Cahn-Hilliard equation; Unconditional energy stability;
Numerical simulations
\end{abstract}

\section{Introduction}
It is well known that LCH equation is one of the phase field models, 
and the history of the phase field model 
can be traced back to a century 
ago which has been applied in many fields
\cite{9,10,11,12,13,14,15,17,18}. 
The Cahn-Hilliard-type models are effective numerical tools 
for simulating interface motions between various materials 
\cite{1,2,3,5,6,7,40,46,52}.
The LCH is the result of the variation of the energy functional
in Sobolev space $H^{-1}$. 
Furthermore, the LCH equation can be regarded as an approximation of 
the NCH model in which the nonlocal convolution
potential is replaced by the differential term \cite{19,20}.
For the nonlocal models, much has been done in mathematical analysis. 
Bates and Han \cite{25,26} analyzed the well-posedness of equations 
with Neumann and Dirichlet boundary conditions. 
Guan et al. pointed out in \cite{29} that the existence and uniqueness of 
periodic solutions of equations can be proved by a similar technique. 
In order to develop a general framework for nonlocal equations, Du 
et al. \cite{28} analyzed a class of nonlocal spread problems with volumetric 
constraint boundary conditions.\par
As the NCH equation gets more and more attention
and is applied in many fields from physics, material science to finance and 
image processing \cite{21,2,23,24,31},
so it is necessary to construct some effective methods for solving the NCH equation.
Due to the functional variation approach used in the modeling process, 
the exact solution of the phase field follows the energy dissipation law, 
which demonstrates the thermodynamic consistency in physics and 
well-posedness in mathematics. 
Therefore, the main challenge of numerical simulation for the NCH equation is 
to design appropriate method to discrete nonlinear and nonlocal terms while 
maintaining the energy stability at the discrete level. 
In addition, if the numerical energy stability has no restriction with respect to
the time step, it is usually called unconditional energy stability \cite{40}.
The significance of energy stability is not only important 
for long time accurate numerical simulation of phase field models, 
but also provides flexibility for dealing with stiffness problems.
This property provides a lot of theoretical and practical support 
for efficient numerical analysis and reliable computer simulation, 
and is widely used in various numerical schemes of classical phase field models,
such as convex splitting schemes \cite{29,30}, 
stabilized schemes \cite{50,51},
the invariant energy quadratization (IEQ) \cite{31}, 
the scalar auxiliary variable (SAV) methods \cite{32}, 
the various variants of SAV \cite{33,41,42} and so on. 
It is also worth studying whether these various effective numerical approaches
can be applied to nonlocal phase field model due to the lack of 
high-order diffusion term \cite{30,34}. \par
Moreover, there is no doubt that, under certain precision requirements, 
if the expected time step is as large as possible, 
the high-order scheme in time is better than the lower-order scheme.
This fact prompted us to develop high-order schemes, 
there are some existing woks, such as the high-order SAV-RK (Runge-Kutta) \cite{61,62},
SAV-GL (general linear time discretization) \cite{63},
implicit-explicit BDF$k$ SAV \cite{64}. 
All of these methods can be used to construct 
high-order schemes for numerical simulation of phase field models.
\par
The purpose of this paper is to establish the high-order linear schemes (in time) for 
the NCH equation and prove the unconditional energy stability of 
the semi-discrete level, which can be naturally extended to 
the fully discrete setting.
We adopt the exponential semi-implicit scalar auxiliary variable (ESI-SAV) 
approach \cite{33}, which is a novel method and has been successfully applied to solve some 
gradient flow and non-gradient but dissipative system.
The basic idea of the ESI-SAV method is to convert the free energy into a 
logarithmic form of a new variable by introducing a scalar auxiliary 
variable (since the exponential form of the original energy is usually 
bounded and positive). 
One of the advantages of this method is that all nonlinear terms 
can be treated explicitly accordingly, 
while linear terms are treated implicitly, 
resulting in decoupled linear systems.
Moreover, whether in continuous or discrete level, 
the newly reconstructed system still maintains the same energy dissipation law as the original system.
On this basis, we establish and analyze the linear 
unconditionally energy stable schemes from the 
first- to the fourth-order (in time) based on ESI-SAV approach.
As shown in detail, they are high-order accurate (third- and fourth-order in time),
easy-to-implement (explicit linear system), 
and unconditionally energy stable (with a discrete energy dissipation law). 
In addition, in terms of spatial discretization, we use the second-order 
central difference formula. 
Considering the huge computational cost and memory requirement of 
solving linear systems generated by the nonlocal terms,
we choose an efficient computational method \cite{32} based on fast Fourier 
transform (FFT) and fast conjugate gradient (FCG) approaches to 
reduce the computational cost.\par
The structure of this paper ia as follows.
In section 2, 
we describe the NCH model with general nonlinear potential and 
the relationship with the LCH model.
In section 3, based on the ESI-SAV approach, we construct the
high-order unconditional energy stable schemes for the NCH model. 
In section 4, we use a variety of the classical numerical examples 
in 2D to verify the accuracy and efficiency of 
our proposed schemes.\par
We use $C$ throughout the paper, with or without subscript, to denote a positive 
constant, which could have different values at different appearances.

\section{The NCH equation}\par
In this section, we will briefly review the NCH equation 
and the connection with the LCH equation.
\subsection{Model formation}
The mesoscopic description of phase transition is usually modeled under the 
assumption that the evolution of order parameter follows the gradient flow of 
free energy relative to a certain measure. 
Now let us consider the bounded spatial domain $\varOmega$,  
a rectangular cell in $\mathbb{R}^d \left( d=1,2,3 \right)$.
The $L^2$ inner product and $L^2$-norm can be defined as
$$
\left( \phi ,\varphi \right) =\int_{\varOmega}{\phi \left( \mathbf{x} \right) \varphi \left( \mathbf{x} \right)d\mathbf{x}},\quad
\lVert \phi \rVert _2=\left( \phi ,\phi \right) ^{\frac{1}{2}},\quad \forall \phi ,\varphi \in L^2\left(\varOmega\right).
$$
The general form of the total 
free energy functional for the NCH equation can be written 
as \cite{34}: 
\begin{equation}
    \begin{aligned}
        \label{eqE1}
        E\left( \phi \right) =\int_{\varOmega}{\left( \frac{\varepsilon ^2}{4}\int_{\varOmega}{J\left( \mathbf{x}-\mathbf{y} \right) \left( \phi \left( \mathbf{x} \right) -\phi \left( \mathbf{y} \right) \right) ^2d\mathbf{y}+F\left( \phi \left( \boldsymbol{x} \right) \right)} \right)}d\mathbf{x},\\
    \end{aligned}
\end{equation}
where $\phi$ is an order parameter, $\varepsilon$ is an 
interface parameter, which represents the interface thickness and satisfies 
$0<\varepsilon \le 1$, $F\left( \phi \right)$ is a nonlinear functional 
and the most commonly used Ginzburg-Landau 
double-well potential has the following form:
\begin{equation}
    F\left( \phi \right) =\frac{\left( \phi ^2-1 \right) ^2}{4}.
    \label{eqF}
\end{equation}
\par
Now we introduce the nonlocal diffusion operator $\mathcal{L}$ 
\cite{29,30,32,34}, which is a linear, 
self-adjoint, positive semi-definite operator:
\begin{equation}
    \label{eqL}
    \mathcal{L} \phi 
    =\int_{\varOmega}{J( \mathbf{x}-\mathbf{y} ) 
    \left( \phi ( \mathbf{x} ) -\phi ( \mathbf{y} 
    ) \right)d\mathbf{y}},\quad \mathbf{x}\in \varOmega.
\end{equation}
\par
The interaction kernel $J$ is integrable and satisfies the following conditions \cite{34}:
\begin{align*} 
& \left( a \right) \,J\left( \mathbf{x} \right) \ge 0 \ \text{for any}\ \mathbf{x} \in \varOmega;\\
& \left( b \right) \,J\left( \mathbf{x} \right) =J\left( -\mathbf{x} \right) ;\\
& \left( c \right) \,J\ \text{is}\ \varOmega-\text{periodic};\\
& \left( d \right) \,J\ \text{is a radial function}.
\end{align*}
\par
Using the conditions of the kernel $J$, we get \cite{28}:
\begin{align*} 
    \left( \mathcal{L} \phi ,
    \phi \right) =\frac{1}{2}\int_{\varOmega}{\int_
    {\varOmega}{J\left( \mathbf{x}-\mathbf{y} \right) \left( \phi \left( \mathbf{x} 
    \right) -\phi \left( \mathbf{y} \right) \right) ^2 d \mathbf{y} d \mathbf{x}}}\ge 0.   
\end{align*}
\par
So the nonlocal total free energy \eqref{eqE1} can be written as:
\begin{align}
    \label{eqE2}
    E\left( \phi \right) =\int_{\varOmega}{\left(\frac{\varepsilon ^2}{2}\phi \mathcal{L}\phi +F\left( \phi \right)\right)}d\mathbf{x}.
\end{align}

\subsection{Connection with the LCH equation}
We now give a brief introduction to the well-known LCH model \cite{1}.
Supposing $\phi \left( \mathbf{x},t \right) :\varOmega \times \left[ 0,T \right]\rightarrow \left[ -1,1 \right]$ 
is an order parameter, the total free energy takes the form \cite{29}:
\begin{equation}
    \label{eqEL}
    E_{local}\left(\phi\right)=\int_{\varOmega}{\left(\frac{1}{2}\varepsilon^2\left|
    \nabla\phi\right|^2+F\left(\phi\right)\right)d\mathbf{x}},
\end{equation}
where $T>0$ is the terminal time,
$\phi =\pm 1$ corresponds to the steady state of a phase transition.\par
We can observe that the difference between the LCH equation and the NCH 
equation is in their free energy. 
The LCH equation is considered to be derived from 
the energy functional \eqref{eqEL} of the conserved gradient flow:
\begin{equation}  
    \left\{\begin{array}{l} 
    \begin{aligned}
        \notag
        \frac{\partial \phi}{\partial t}={}&M\Delta \mu,\\
        \mu ={}&-\varepsilon ^2\Delta \phi +f\left( \phi \right),
    \end{aligned}
    \end{array}\right. 
\end{equation}    
where $M$ is the mobility constant, 
$f\left( \phi \right) =F'\left( \phi \right)$ is a nonlinear term, 
and $\mu$ is the chemical potential. 
The unknown $\phi\left(\mathbf{x,}t\right)$ is subject to the initial condition 
$\phi \left( \mathbf{x},0 \right) =\phi _0\left( \mathbf{x} \right) ,\mathbf{x}\in \varOmega $, 
with the following boundary condition: periodic, or no flux boundary condition which is given as :
$\frac{\partial \phi}{\partial \vec{\mathbf{n}}}=\frac{\partial \mu}{\partial \vec{\mathbf{n}}}=0$ on $\partial \varOmega $,
where $\vec{\mathbf{n}}$ is the unit outward normal vector on $\partial \varOmega$.
\par
Similarly, using the variational principle for the nonlocal
energy functional \eqref{eqE2}, we can get
\begin{equation}  
        \left\{\begin{array}{l} 
        \begin{aligned}
            \label{NN}
            \frac{\partial \phi}{\partial t} ={}& M\Delta \mu,\\
            \mu ={}& \varepsilon ^2\mathcal{L}\phi +f\left( \phi \right).
        \end{aligned}
        \end{array}\right. 
    \end{equation}
    \par 
Next, we briefly introduce how the NCH model is derived from the functional variation in the 
free energy functional \eqref{eqE2}. 
Denoting its variational derivative as $\mu =\frac{\delta E}{\delta \phi}$,
the general form of the gradient flow model can be written as \cite{35}
\begin{equation}
    \frac{\partial \phi}{\partial t}=M\mathcal{G}\mu,\quad \left( \mathbf{x},t \right) \in \varOmega \times \left[ 0,T \right]\rightarrow \left[ -1,1 \right],
    \label{eq2.5}
\end{equation}
where the $\mathcal{G}$ is a linear, negative semi-definite operator.
\par
The specific form of the chemical potential $\mu$ can 
be derived by solving the variational derivative of the functional \eqref{eqE2},
and the process is as follows:\\
for any function $\eta \left( \mathbf{x} \right)$, which is sufficiently smooth and satisfies
$\eta |_{\partial \varOmega}=0$ on $\partial \varOmega $, we have
\begin{align*}
    \int_{\varOmega}{\frac{\delta E\left( \phi \right)}{\delta \phi}\eta}d\mathbf{x}
    ={}& \lim_{\theta \rightarrow 0}\frac{E\left( \phi +\theta \eta \right) -E\left( \phi \right)}{\theta}\\    
    ={}& \lim_{\theta \rightarrow 0}\frac{\varepsilon ^2}{2\theta}\int_{\varOmega}{\left[ \left( \mathcal{L}\left( \phi +\theta \eta \right) ,\phi +\theta \eta \right) -\left( \mathcal{L}\phi ,\phi \right) \right]}d\mathbf{x}\\
    +{}& \lim_{\theta \rightarrow 0}\frac{1}{\theta}\int_{\varOmega}{\left[ F\left( \phi +\theta \eta \right) -F\left( \phi \right) \right]}d\mathbf{x}\\
    ={}& \varepsilon ^2\int_{\varOmega}{\left( \mathcal{L}\phi \right) \eta}d\mathbf{x}+\int_{\varOmega}{F'\left( \phi \right) \eta}d\mathbf{x}\\
    ={}& \int_{\varOmega}{\left( \varepsilon ^2\mathcal{L}\phi +f\left( \phi \right) \right) \eta}d\mathbf{x},
\end{align*}
thus,
\begin{equation}
    \mu = \frac{\delta E\left( \phi \right)}{\delta \phi}=\varepsilon ^2 \mathcal{L} \phi +f\left( \phi \right).
    \label{eq2.6}
\end{equation}
\par
Then, by combining equations \eqref{eq2.5}, \eqref{eq2.6}, 
we obtain the NCH model: 
\begin{equation}  
    \label{eqGL}
    \left\{\begin{array}{l} 
    \begin{aligned}
        &\frac{\partial\phi}{\partial t}=M\mathcal{G}\mu,
        &\left(\mathbf{x},t\right) \in \varOmega _T,\\
        &\mu = \varepsilon ^2\mathcal{L}\phi +f\left( \phi \right),
        &\left( \mathbf{x},t \right) \in \varOmega _T,\\
    \end{aligned}
    \end{array}\right. 
\end{equation}
corresponding to initial condition and periodic or no flux boundary conditions.
In system \eqref{NN}, we take $\mathcal{G}$ as $\Delta $.
If we define $\mathcal{G}$ as the nonlocal diffusion operator $-\mathcal{L}$,  
we can get a general NCH equation with general nonlinear potential \cite{32}:
\begin{equation}  
    \left\{\begin{array}{l} 
        \begin{aligned}
            \label{NCH}
            \frac{\partial \phi}{\partial t}={}&-M\mathcal{L}\mu,\\
            \mu ={}& \varepsilon ^2\mathcal{L}\phi +f\left( \phi \right).
        \end{aligned}
        \end{array}\right. 
    \end{equation}
\par
Furthermore, in order to explain the relationship between 
the local and nonlocal energies,
we can use Taylor expansion \cite{37,38} to approximate the interaction energy density 
with the periodicity of $\phi$ and $J$, and the conditions $\left(a\right)-\left( d \right)$ of $J$,
then we can get 
\begin{align*}
    &\frac{\varepsilon^2}{4}\int_{\varOmega}{J\left(\mathbf{x}-\mathbf{y}\right) 
    \left(\phi\left(\mathbf{x}\right)-\phi\left(\mathbf{y}\right)\right)^2d\mathbf{y}}\\
    =&\frac{\varepsilon^2}{4}\int_{\varOmega}{J\left(\mathbf{y}\right) 
    \left(\phi\left(\mathbf{x}\right)-\phi\left(\mathbf{x}+\mathbf{y}\right) 
    \right)^2d\mathbf{y}}\\
    \approx&\frac{\varepsilon^2}{4}\int_{\varOmega}{J\left(\mathbf{y}\right)\left| 
    \mathbf{y}\right|^2\left|\nabla\phi\left(\mathbf{x}\right)\right|^2d\mathbf{y}}\\
    =&\frac{\varepsilon ^2}{2}\left|\nabla\phi\left(\mathbf{x}\right)\right|^2.    
\end{align*}
\par
The operator $-\mathcal{L}$ of the NCH model \eqref{NCH} is negative semi-definite, 
so the free energy equation \eqref{eqE2} is decreasing with time,
\begin{equation}
    \begin{aligned}
        \label{eqEss}
        \frac{dE\left( \phi \left( t \right) \right)}{dt}
        ={}& \frac{d}{dt}\left( \frac{\varepsilon ^2}{2}\left( \mathcal{L}\phi ,\phi \right) +\int_{\varOmega}{F\left( \phi \right)}d\mathbf{x} \right)\\
        ={}& \left( \phi _t,\varepsilon ^2\mathcal{L}\phi \right) +\left( \phi _t,f( \phi ) \right) \\
        ={}& \left( \phi _t,\varepsilon ^2\mathcal{L}\phi +f\left( \phi \right) \right) \\
        ={}& \left(-M\mathcal{L}\mu,\mu \right) \le 0.
    \end{aligned} 
\end{equation}

\section{The ESI-SAV scheme for the NCH equation}\par
In this section, we will construct and analyze the linear high-order semi-discrete 
time marching numerical schemes for the NCH model 
with general nonlinear potential. 
It will be clear that the unconditional energy stability of 
the semi-discrete scheme is also valid in the fully discrete formulation.
Let  $N_t>0$ be a positive integer, and
$$ T=N_t\varDelta t, \quad t^n=n\varDelta t,\quad  \forall n\le N_t. $$
For the ESI-SAV approach, similar to \cite{33}, 
we introduce an exponential scalar auxiliary variable :
\begin{equation*}
    R\left( t \right) =\exp \left( E\left( \phi \right) \right) =\exp \left( \frac{\varepsilon ^2}{2}\left( \mathcal{L}\phi ,\phi \right) +\int_{\varOmega}{F\left( \phi \right) d\mathbf{x}} \right) >0, \quad \forall t\in \left[ 0,T \right] .
\end{equation*}
\par
Then the equivalent new format of the NCH model \eqref{NCH} is:
\begin{equation} 
    \label{NCH0}   
    \left\{\begin{array}{l}
        \begin{aligned}            
            \frac{\partial \phi}{\partial t}={}&-M\mathcal{L}\mu,\\
            \mu ={}&\varepsilon ^2\mathcal{L}\phi +\frac{R}{\exp 
            \left(E\left( \phi \right) \right)}f\left( \phi \right),\\
            \frac{dR}{dt}={}&R\left( -M\mathcal{L}\mu ,\mu \right).\\       
        \end{aligned}  
    \end{array} \right.
\end{equation} 
\par
For the above system, the new energy is defined as $\ln{R}$, 
and it can be observed that the new energy is equivalent 
to the original energy, that is $\ln{R}=E$,
and satisfies the energy dissipation law.
Rewrite the third equation of the above system \eqref{NCH0}, we get
\begin{align*}
   \frac{dE}{dt}= \frac{d\ln R}{dt}=\frac{1}{R}\frac{dR}{dt}=\left( -M\mathcal{L}\mu ,\mu \right) \le 0.
\end{align*}
This result is consistent with the energy dissipation law of the original energy \eqref{eqEss}.

\subsection{First-order ESI-SAV scheme}
The frist-order ESI-SAV scheme for \eqref{NCH0} reads as:
\begin{equation} 
    \label{NCH1}   
    \left\{ \begin{array}{l}
        \begin{aligned}            
	        \frac{\phi ^{n+1}-\phi ^n}{\varDelta t}={}&-M\mathcal{L}\mu ^{n+1},\\
	        \mu ^{n+1}={}&\varepsilon ^2\mathcal{L}\phi ^{n+1}+\frac{R^{n+1}}{\exp \left( E\left( \phi ^n \right) \right)}f\left( \phi ^n \right),\\
	        \frac{R^{n+1}-R^n}{\varDelta t}={}&R^{n+1}\left( -M\mathcal{L}\bar{\mu}^n,\bar{\mu}^n \right),\\
        \end{aligned}  
    \end{array} \right.
\end{equation}  
with the initial conditions
\begin{align*}
    &\phi ^0=\phi _0\left( \mathbf{x}  \right),\quad \forall \mathbf{x}\in \varOmega,\\
    &R^0=\exp \left( E\left( \phi ^0 \right) \right).
\end{align*}
\begin{theorem}
    \label{the-1}
    The scheme \eqref{NCH1} for the NCH equation is 
    unconditionally energy stable in the sense that
    \begin{equation*} 
        R^{n+1}-R^n=-R^{n+1}\varDelta t\left( M\mathcal{L}\bar{\mu}^n,\bar{\mu}^n \right) \le 0,
    \end{equation*}
    \\
    and more importantly we have
    \begin{equation*} 
        \ln R^{n+1}-\ln R^n\le 0.
    \end{equation*}
\end{theorem}
\begin{proof}
    By the definition of $R\left( t \right)$ and the negative semi-definite operator
    $-\mathcal{L}$, we obtain 
    $R^{n+1}=\dfrac{R^n}{1+\varDelta t\left( M\mathcal{L}\bar{\mu}^n,\bar{\mu}^n \right)}>0$.
    Then combining the inequality $R^{n+1}>0$ with the third equation of \eqref{NCH1}, 
    we can easily obtain the following modified energy stability:
    \begin{equation*} 
        R^{n+1}-R^n=R^{n+1}\varDelta tM\left( \mathcal{G}\bar{\mu}^n,\bar{\mu}^n \right) \le 0.
    \end{equation*}
    \par
    Noting that $E\left( \phi \right) =\ln \left( \exp \left( E\left( \phi \right) \right) \right) =\ln R$ 
    and the logarithm function is strictly monotonically increasing, we can also obtain the following energy stability:
    \begin{equation*} 
        \ln R^{n+1}-\ln R^n\le 0.
    \end{equation*}
\end{proof}
\par
Next, in order to show the computational simplicity and effectiveness of the proposed 
numerical scheme, we give the explicit process for solving 
the first-order scheme \eqref{NCH1} as below, 
and other high-order schemes are solved similarly.
By rewriting the scheme \eqref{NCH1}, we can get 
\begin{subequations}
    \label{eqNCH1}
    \begin{align}
        \left( \frac{1}{\varDelta t}+M\varepsilon ^2\mathcal{L}^2 \right) \phi ^{n+1}&=\frac{1}{\varDelta t}\phi ^n-M\mathcal{L}\frac{R^{n+1}}{\exp \left( E\left( \phi ^n \right) \right)}f\left( \phi ^n \right) \label{eq1_1},\\
        R^{n+1}&=\frac{R^n}{1+\varDelta t\left( M\mathcal{L}\bar{\mu}^n,\bar{\mu}^n \right)}. \label{eq1_2}       
        \end{align}  
    \end{subequations}
\par
Obviously, \eqref{eqNCH1} is uniquely solvable for any $\varDelta t >0$ since 
$ \frac{1}{\varDelta t}+M\varepsilon ^2\mathcal{L}^2 $ is 
positive definite.
To summarize, we implement \eqref{NCH1} as follows:\\
with the value of $\phi^{n}$ known,
\begin{enumerate}[itemindent=1em]
    \item[($i$)] Compute $\bar{\mu}^n$ by $\bar{\mu}^n=\varepsilon ^2\mathcal{L}\phi ^n+f\left( \phi ^n \right)$,
    \item[($ii$)] Compute $R^{n+1}$ from equation \eqref{eq1_2},
    \item[($iii$)] Compute $\phi^{n+1} $ from equation \eqref{eq1_1}.
\end{enumerate}
\par
Note that in the above steps $(ii),(iii)$, we only need to solve explicitly 
two linear equations with constant coefficients of the form 
\begin{equation*}
    \bar{A}\bar{\phi}=\bar{b}.
\end{equation*}

\subsection{Second-order Crank-Nicolson scheme }
In this section, we will present and analyze linear, 
second-order (in time) numerical 
ESI-SAV scheme based on the Crank-Nicolson formula. 
We firstly introduce a new variable 
$\xi =\frac{R}{\exp \left( E\left( \phi \right) \right)}$.  
It is obviously that $\xi \equiv$ 1 at the continuous level, 
and $\xi \left( 2-\xi \right)$ is also 
equal to 1. So the system \eqref{NCH} can be rewritten as:
\begin{equation} 
    \label{NCH00}   
    \left\{ \begin{array}{l}
        \begin{aligned}            
	        \frac{\partial \phi}{\partial t}={}&-M\mathcal{L}\mu,\\
	        \mu ={}&\varepsilon ^2\mathcal{L}\phi +\xi \left( 2-\xi \right) f\left( \phi \right),\\
	        \xi ={}&\frac{R}{\exp \left( E\left( \phi \right) \right)},\\
            \frac{dR}{dt}={}&R\left( -M\mathcal{L}\mu ,\mu \right).
        \end{aligned}  
    \end{array} \right.
\end{equation} 
\par
The new equivalent system also maintains the original energy dissipation law:
\begin{equation*}
    \frac{dE}{dt}=\frac{d \left( \ln R \right)}{dt}=\frac{1}{R}\frac{dR}{dt} =\left( -M\mathcal{L}\mu ,\mu \right) \le 0.
\end{equation*}
\par
A second-order ESI-SAV scheme for the system \eqref{NCH00} reads as:  
\begin{equation} 
    \label{NCHCN}
    \left\{ \begin{array}{l}
        \begin{aligned}
	    &\frac{\phi ^{n+1}-\phi ^n}{\varDelta t}=-M\mathcal{L}\mu ^{n+\frac{1}{2}},\\
	    &\mu ^{n+\frac{1}{2}}=\varepsilon ^2\mathcal{L}\frac{\phi ^{n+1}+\phi ^n}{2}+V\left( \xi ^{n+1} \right) f\left( \phi ^{*,n+\frac{1}{2}} \right),\\
	    &\xi ^{n+1}=\frac{R^{n+1}}{\exp \left( E\left( \phi ^{*,n+\frac{1}{2}} \right) \right)},\\
	    &V\left( \xi ^{n+1} \right) =\xi ^{n+1}\left( 2-\xi ^{n+1} \right),\\
	    &\frac{R^{n+1}-R^n}{\varDelta t}=R^{n+1}\left( -M\mathcal{L}\bar{\mu}^{n+\frac{1}{2}},\bar{\mu}^{n+\frac{1}{2}} \right),\\
        \end{aligned}        
    \end{array} \right. 
\end{equation}
with the initial conditions
\begin{align*}
    &\phi ^0=\phi _0\left( \mathbf{x}\right),\quad \forall \mathbf{x}\in \varOmega,\\
    &R^0=\exp \left( E\left( \phi ^0 \right) \right),       
\end{align*}
where $\bar{\mu}^{n+\frac{1}{2}}=\varepsilon ^2\mathcal{L}\phi ^{*,n+\frac{1}{2}}+f\left( \phi ^{*,n+\frac{1}{2}} \right) ,$
$\phi ^{*,n+\frac{1}{2}}$ is any explicit $O\left( \varDelta t^2 \right)$ 
approximation of $\phi \left( t^{n+\frac{1}{2}} \right)$. We can choose the
extrapolation formula
\begin{align*}
    \phi ^{*,n+\frac{1}{2}}=\frac{3}{2}\phi ^n-\frac{1}{2}\phi ^{n-1},\quad n\ge 1,
\end{align*}
\par
and for $\phi ^{*,\frac{1}{2}}$, we can obtain it as follows:
\begin{equation*}
    \frac{\phi ^{*,\frac{1}{2}}-\phi ^0}{\varDelta t/2}=-M\mathcal{L}\left( \varepsilon ^2\mathcal{L}\phi ^{*,\frac{1}{2}}+f\left( \phi ^0 \right) \right). 
\end{equation*}
\par
Similar to the proof of Theorem \ref{the-1}, we can easily get the following theorem:
\begin{theorem}
    The scheme \eqref{NCHCN} for the NCH equation is unconditionally energy stable in the sense that
    \begin{equation}
        R^{n+1}-R^n=R^{n+1}\varDelta t\left( -M\mathcal{L}\bar{\mu}^{n+\frac{1}{2}},\bar{\mu}^{n+\frac{1}{2}} \right) \le 0,
    \end{equation}
    and more importantly we have 
    \begin{equation}
        \ln R^{n+1}-\ln R^n\le 0.
    \end{equation}
\end{theorem} 

\subsection{High-order BDF$k$ schemes}
We use the ESI-SAV method combined with $k$-step BDF (BDF$k$) to extend 
the high-order unconditional energy stability scheme for the NCH equation.
Now firstly we rewrite the system \eqref{NCH} by introducing a new function 
$V\left( \xi \right)$ which is equal to $1$ at the continuous level:
\begin{equation}
    \label{NCH000}
    \left\{ \begin{array}{l}
    \begin{aligned}
	&\frac{\partial \phi}{\partial t}=-M\mathcal{L}\mu ,\\
	&\mu =\varepsilon ^2\mathcal{L}\phi +V\left( \xi \right) f\left( \phi \right) ,\\
	&\xi =\frac{R}{\exp \left( E\left( \phi \right) \right)},\\
	&\frac{dR}{dt}=R\left( -M\mathcal{L}\mu ,\mu \right) .\\
    \end{aligned}
\end{array} \right.
\end{equation}
\par
The new equivalent system also keeps the original energy dissipation law:
\begin{equation*}
    \frac{dE}{dt}=\frac{d\ln \left( R \right)}{dt}=\frac{1}{R}\frac{dR}{dt}=\left( -M\mathcal{L}\mu ,\mu \right) \le 0.
\end{equation*}
\par
Denoting $V\left( \xi \right) =\xi \left( 2-\xi \right) $, a 
second-order ESI-SAV scheme based on the BDF2 formula 
for \eqref{NCH000} reads as:\\
for $n\ge 1$,
\begin{equation}
    \label{NCH2}
    \left\{ \begin{array}{l}
    \begin{aligned}
        &\frac{3\phi ^{n+1}-4\phi ^n+\phi ^{n-1}}{2\varDelta t}=-M\mathcal{L}\mu ^{n+1},\\
        &\mu ^{n+1}=\varepsilon ^2\mathcal{L}\phi ^{n+1}+V\left( \xi ^{n+1} \right) f\left( \phi ^{*,n+1} \right),\\
        &V\left( \xi ^{n+1} \right) =\xi ^{n+1}\left( 2-\xi ^{n+1} \right),\\
        &\xi ^{n+1}=\frac{R^{n+1}}{\exp \left( E\left( \phi ^{*,n+1} \right) \right)},\\
        &\frac{R^{n+1}-R^n}{\varDelta t}=R^{n+1}\left( -M\mathcal{L}\bar{\mu}^{n+1},\bar{\mu}^{n+1} \right),\\
    \end{aligned}
\end{array} \right.
\end{equation}
with the initial conditions
\begin{align*}
    &\phi ^0=\phi _0\left( \mathbf{x} \right),\quad \forall \mathbf{x}\in \varOmega,\\
    &R^0=\exp \left( E\left( \phi ^0 \right) \right),      
\end{align*}
where $\bar{\mu}^{n+1}=\varepsilon ^2\mathcal{L}\phi ^{*,n+1}+f\left( \phi ^{*,n+1} \right),$
$\phi ^{*,n+1}$ is any explicit $O\left( \varDelta t^2 \right)$ approximation 
of $\phi \left( t^{n+1} \right)$, so it can be calculated by
$$\phi ^{*,n+1}=2\phi ^n-\phi ^{n-1},\quad n\ge 1.$$
\par
Denoting $V\left( \xi \right) =\xi \left( 3-3\xi +\xi ^2 \right)$,
a third-order ESI-SAV scheme based on the BDF3 formula for \eqref{NCH000} 
reads as:\\ 
for $n\ge 2$,
\begin{equation}
    \label{NCH3}
    \left\{ \begin{array}{l}
    \begin{aligned}
	    &\frac{11\phi ^{n+1}-18\phi ^n+9\phi ^{n-1}-2\phi ^{n-2}}{6\varDelta t}=-M\mathcal{L}\mu ^{n+1},\\
	    &\mu ^{n+1}=\varepsilon ^2\mathcal{L}\phi ^{n+1}+V\left( \xi ^{n+1} \right) f\left( \phi ^{*,n+1} \right),\\
	    &V\left( \xi ^{n+1} \right) =\xi ^{n+1}\left( 3-3\xi ^{n+1}+\left( \xi ^{n+1} \right) ^2 \right),\\
	    &\xi ^{n+1}=\frac{R^{n+1}}{\exp \left( E\left( \phi ^{*,n+1} \right) \right)},\\
	    &\frac{R^{n+1}-R^n}{\varDelta t}=R^{n+1}\left( -M\mathcal{L}\bar{\mu}^{n+1},\bar{\mu}^{n+1} \right),\\
    \end{aligned}
\end{array} \right.
\end{equation}
where $\bar{\mu}^{n+1}=\varepsilon ^2\mathcal{L}\phi ^{*,n+1}+f\left( \phi ^{*,n+1} \right),$
$\phi ^{*,n+1}$ is any explicit $O\left( \varDelta t^3 \right)$ approximation 
of $\phi \left( t^{n+1} \right)$, and can be calculated by
$$\phi ^{*,n+1}=3\phi ^n-3\phi ^{n-1}+\phi ^{n-2},\quad n\ge 2.$$
\par
Denoting $V\left( \xi \right) =\xi \left( 2-\xi \right) \left( 2-2\xi +\xi ^2 \right)$,
a fourth-order ESI-SAV scheme based on the BDF4 formula for \eqref{NCH000} 
reads as:\\
for $n\ge 3$,
\begin{equation}
    \label{NCH4}
    \left\{ \begin{array}{l}
    \begin{aligned}
	&\frac{25\phi ^{n+1}-48\phi ^n+36\phi ^{n-1}-16\phi ^{n-2}+3\phi ^{n-3}}{12\varDelta t}=-M\mathcal{L}\mu ^{n+1},\\
	&\mu ^{n+1}=\varepsilon ^2\mathcal{L}\phi ^{n+1}+V\left( \xi ^{n+1} \right) f\left( \phi ^{*,n+1} \right),\\
	&V\left( \xi ^{n+1} \right) =\xi ^{n+1}\left( 2-\xi ^{n+1} \right) \left( 2-2\xi ^{n+1}+\left( \xi ^{n+1} \right) ^2 \right),\\
	&\xi ^{n+1}=\,\,\frac{R^{n+1}}{\exp \left( E\left( \phi ^{*,n+1} \right) \right)},\\
	&\frac{R^{n+1}-R^n}{\varDelta t}=R^{n+1}\left( -M\mathcal{L}\bar{\mu}^{n+1},\bar{\mu}^{n+1} \right),\\
    \end{aligned}
\end{array} \right.
\end{equation}
where $\bar{\mu}^{n+1}=\varepsilon ^2\mathcal{L}\phi ^{*,n+1}+f\left( \phi ^{*,n+1} \right)$,
$\phi ^{*,n+1}$ is any explicit $O\left( \varDelta t^4 \right)$ approximation 
of $\phi \left( t^{n+1} \right)$, and can be calculated by
$$\phi ^{*,n+1}=4\phi ^n-6\phi ^{n-1}+4\phi ^{n-2}-\phi ^{n-3},\quad n\ge 3.$$
\par
Similar to the proof of Theorem \ref{the-1}, we can easily get the following theorem:
\begin{theorem}
    The schemes \eqref{NCH2}-\eqref{NCH4} for the NCH equation is unconditionally energy stable in the sense that
    \begin{equation*}
        R^{n+1}-R^n=R^{n+1}\varDelta t\left( -M\mathcal{L}\bar{\mu}^{n+1},\bar{\mu}^{n+1} \right) \le 0,
    \end{equation*}
    \\
    and more importantly we have 
    \begin{equation*}
        \ln R^{n+1}-\ln R^n\le 0.
    \end{equation*}
\end{theorem}

\subsection{The improvement of the schemes}
In order to prevent the instability of the result caused by the rapid growth 
of the exponential function, and ensure the energy dissipation law, 
we can add a positive number $C$ large enough to 
redefine the exponential scalar auxiliary variable:
\begin{equation}
    \label{DR}
    R\left(t\right)=\exp\left(\frac{E\left(\phi\right)}{C}\right) 
    =\exp\left(\frac{\varepsilon^2}{2C}\left(\mathcal{L}\phi,\phi\right) 
    +\frac{1}{C}\int_{\varOmega}{F\left(\phi \right)d\mathbf{x}}\right).
\end{equation}
\par
Then the scheme \eqref{NCH000} can be improved to the following equivalent format:
\begin{equation*}
    \left\{ \begin{array}{l}
    \begin{aligned}
        &\frac{\partial \phi}{\partial t}=-M\mathcal{L}\mu ,\\
        &\mu =\varepsilon ^2\mathcal{L}\phi +V\left( \xi \right) f\left( \phi \right) ,\\
        &\xi =\frac{R}{\exp \left( \dfrac{E\left( \phi \right)}{C} \right)},\\
        &\frac{dR}{dt}=\frac{R}{C}\left( -M\mathcal{L}\mu ,\mu \right) .\\
    \end{aligned}
\end{array} \right.
\end{equation*}
\par
According to formula \eqref{DR}, the modified free energy can be defined as $ClnR$,
furthermore, it satisfies $ClnR=E$ at the continuous level.
The new equivalent system also keeps the original energy dissipation law:
\begin{equation*}
    \frac{dE}{dt}=\frac{d\left( C\ln R \right)}{dt}=\frac{C}{R}\frac{dR}{dt} =-M\left( \mathcal{L}\mu ,\mu \right) \le 0.
\end{equation*}

\subsection{Spatial discretization}
We use the second-order central finite difference formula to discretize the 
spatial operator. Let $N_x$,$N_y$ be positive integers and the two dimensional domain
$\varOmega =\left[ -L_x,L_x \right] \times \left[ -L_y,L_y \right] $.
For simplicity of explanation, 
we consider periodic boundary conditions for the NCH equation.
We divide the domain into rectangular meshes with mesh size 
\begin{equation*}
    h_x=\dfrac{2L_x}{N_x},h_y=\dfrac{2L_y}{N_y}.
\end{equation*}
So we define the following uniform grid:
\begin{equation*}
    \varOmega_h=\left\{ \left( x_i,y_j \right) \left| x_i=-L_x+ih_x,y_j=-L_y+jh_y,0\le i\le N_x,0\le j\le N_y \right. \right\}.
\end{equation*}
From \cite{33,34}, we can write the nonlocal operator $\mathcal{L}$ in \eqref{eqL} 
as the following equivalent form:
\begin{equation*}
    \mathcal{L}\phi =\left( J*1 \right) \phi -J*\phi,
\end{equation*}
where
\begin{align*}
    &J*1=\int_{\varOmega}{J\left(\mathbf{x}\right)d\mathbf{x}},\\
    &\left(J*\phi\right)\left(\mathbf{x}\right)
    =\int_{\varOmega}{J\left(\mathbf{x}-\mathbf{y}\right)}\phi\left(\mathbf{y}\right)d\mathbf{y}
    =\int_{\varOmega}{J\left(\mathbf{y}\right)}\phi\left(\mathbf{x}-\mathbf{y}\right)d\mathbf{y},
\end{align*}
are exactly the periodic convolutions.
\par
For any $\phi$, we can discrete $\mathcal{L}\phi$ at $\left( t^n,x_i,y_j \right),
\left( x_i,y_j \right) \in \varOmega _h$, as follows:
\begin{equation}
    \label{eqLL}
    \left( \mathcal{L}_h\phi \right) _{i,j}^{n}=\left( J*1 \right) _{i,j}\phi _{i,j}^{n}-\left( J*\phi \right) _{i,j}^{n},
\end{equation}
where
\begin{equation*} 
    \begin{aligned}            
        \left( J*1 \right) _{i,j}\phi _{i,j}^{n}=& h_xh_y\left[ \sum_{m_1=1}^{N_x-1}{\sum_{m_2=1}^{N_y-1}{J\left( x_{m_1}-x_i,y_{m_2}-y_j \right)}} \right.\\ 
        &+\frac{1}{2}\sum_{m_1=1}^{N_x-1}{\left( J\left( x_{m_1}-x_i,y_0-y_j \right) +J\left( x_{m_1}-x_i,y_{N_y}-y_j \right) \right)}\\
        &+\frac{1}{2}\sum_{m_2=1}^{N_y-1}{\left( J\left( x_0-x_i,y_{m_2}-y_j \right) +J\left( x_{N_x}-x_i,y_{m_2}-y_j \right) \right)}\\
        &+\frac{1}{4}\left( J\left( x_0-x_i,y_0-y_j \right) +J\left( x_{N_x}-x_i,y_0-y_j \right) \right)\\
        &\left. +\frac{1}{4}\left( J\left( x_0-x_i,y_{N_y}-y_j \right) +J\left( x_{N_x}-x_i,y_{N_y}-y_j \right) \right) \right] \phi _{i,j}^{n},       
   \end{aligned}  
\end{equation*}
and
\begin{equation*} 
    \begin{aligned}            
        \left( J*\phi \right)_{i,j}^{n}=&h_xh_y\left[ \sum_{m_1=1}^{N_x-1}{\sum_{m_2=1}^{N_y-1}{J\left( x_{m_1}-x_i,y_{m_2}-y_j \right) \phi _{m_1,m_2}^{n}}} \right.\\ 
        &+\frac{1}{2}\sum_{m_1=1}^{N_x-1}{\left( J\left( x_{m_1}-x_i,y_0-y_j \right) \phi _{m_1,0}^{n}+J\left( x_{m_1}-x_i,y_{N_y}-y_j \right) \phi _{m_1,N_y}^{n} \right)}\\
        &+\frac{1}{2}\sum_{m_2=1}^{N_y-1}{\left( J\left( x_0-x_i,y_{m_2}-y_j \right) \phi _{0,m_2}^{n}+J\left( x_{N_x}-x_i,y_{m_2}-y_j \right) \phi _{N_x,m_2}^{n} \right)}\\
        &+\frac{1}{4}\left( J\left( x_0-x_i,y_0-y_j \right) \phi _{0,0}^{n}+J\left( x_{Nx}-x_i,y_0-y_j \right) \phi _{N_x,0}^{n} \right)\\
        &\left. +\frac{1}{4}\left( J\left( x_0-x_i,y_{N_y}-y_j \right) \phi _{0,N_y}^{n}+J\left( x_{N_x}-x_i,y_{N_y}-y_j \right) \phi _{N_x,N_y}^{n} \right) \right].
   \end{aligned}  
\end{equation*}
\par
Combining the fourth-order system \eqref{NCH4}, 
and the discrete nonlocal operator expression \eqref{eqLL},
we obtain the corresponding fully discrete finite difference scheme 
as follows:
\begin{equation} 
    \label{A_4}
    \left\{ \begin{array}{l}
    \begin{aligned}            
	    &\frac{25\phi _{i,j}^{n+1}-48\phi _{i,j}^{n}+36\phi _{i,j}^{n-1}-16\phi _{i,j}^{n-2}+3\phi _{i,j}^{n-3}}{\varDelta t}=-M\left(\mathcal{L}_h\mu \right) _{i,j}^{n+1},\\
	    &\mu _{i,j}^{n+1}=\varepsilon ^2\left( \mathcal{L}_h\phi \right) _{i,j}^{n+1}+V\left( \xi ^{n+1} \right) f\left( \phi _{i,j}^{*,n+1} \right),\\
	    &V\left( \xi ^{n+1} \right) =\xi ^{n+1}\left( 2-\xi ^{n+1} \right) \left( 2-2\xi ^{n+1}+\left( \xi ^{n+1} \right) ^2 \right),\\
	    &\xi ^{n+1}=\ \frac{R^{n+1}}{\exp \left( E(\phi ^{*,n+1} )\right)},\\
        &\frac{R^{n+1}-R^n}{\varDelta t}=R^{n+1}\sum_{i=0}^{N_x}{\sum_{j=0}^{N_y}{M\left( \mathcal{G}_h\bar{\mu} \right) _{i,j}^{n+1}\bar{\mu}_{i,j}^{n+1}}},
    \end{aligned}  
\end{array} \right.
\end{equation}
where
\begin{align*}
    &\bar{\mu}_{i,j}^{n+1}=\varepsilon ^2\left( \mathcal{L}_h\phi \right) _{i,j}^{*,n+1}+f\left( \phi _{i,j}^{*,n+1} \right) ,\quad n\ge 3,0\le i\le N_x,0\le j\le N_y,\\
    &\phi _{i,j}^{*,n+1}=4\phi _{i,j}^{n}-6\phi _{i,j}^{n-1}+4\phi _{i,j}^{n-1}-\phi _{i,j}^{n-3},\quad n\ge 3,0\le i\le N_x,0\le j\le N_y. 
\end{align*}    
\par
In order to monitor the change of free energy during evolution, 
the original discrete energy functional is defined as:
\begin{equation}
    \label{EEE}
    E^{n}=\sum_{i=0}^{N_x}{\sum_{j=0}^{N_y}{F\left( \phi _{i,j}^{n} \right) +}}\frac{\varepsilon ^2}{2}\sum_{i=0}^{N_x}{\sum_{j=0}^{N_y}{ \left( \mathcal{L}_h\phi \right) _{i,j}^{n}\cdot \phi _{i,j}^{n} }}.
\end{equation} 
In addition, combined with the $R^n$ of scheme \eqref{A_4}, 
the modified discrete energy functional can be defined as:
\begin{equation}
    \label{EE}
    \bar{E}^n=C\ln \left( R^n \right).
\end{equation} 
\par
We can see that the nonlocal diffusion term will lead the stiffness maxtrix to be 
almost full matrix which requires huge computational work and large memory. 
In this case, the fast solution method for solving the derived linear system will 
become very important and necessary. Due to the special structure of 
coefficient matrices, we adopt the similar methods as \cite{32} to ensure the efficiency 
of the algorithm based on the Fast Fourier transform (FFT) and the fast conjugate 
gradient (FCG) method. 
We set $N_x=N_y=N$, 
the overall computational cost of the FCG method is $O\left( N\log ^2N \right)$ 
with the number of iterations $O\left( \log N \right) $, 
instead of $O\left( N^3 \right) $ by using the Gaussian elimination method,
and the fast solver will reduce memory requirement from $O\left( N^2 \right) $ to $O\left( N \right)$.
\section{Numerical simulations}
In this section, various simulations are provided to verify  
the theoretical results of the developed numerical schemes. \par
This smooth kernel is given by the Gaussian function in the following form \cite{34}:
\begin{equation}
    \label{eqJ}
    \begin{aligned}
        J_{\delta}=\frac{4}{\pi ^{d/2}\delta ^{d+2}}e^{-\frac{\left| \mathbf{x} \right|^2}{\delta ^2}},
        \quad \mathbf{x}\in \varOmega \subset \mathbb{R}^d,\quad \delta >0,
    \end{aligned}
\end{equation}
and it satisfies
\begin{align*}
    &\int_{\mathbb{R}^d}{J_{\delta}\left(\mathbf{x}\right)}\,\text{d}\mathbf{x}=\frac{4}{\delta ^2},\\
    &\int_{\mathbb{R}^d}{J_{\delta}\left(\mathbf{x}\right)}\left|\mathbf{x}\right|\,\text{d}\mathbf{x}=2d.
\end{align*}
\par
From \cite{34}, this kernel $J_\delta $ has the proposition:
for any $\phi \in C^{\infty}\left( \varOmega \right)$ , 
$\mathbf{x}\in \varOmega,\, \mathcal{L}_{\delta}\phi \left( \mathbf{x} \right) 
\rightarrow -\varDelta \phi \left( \mathbf{x} \right)$, 
as $\delta \rightarrow 0$.
It tells us that the NCH equation with Gaussian kernel converges to the LCH equation 
when $\delta$ approaches zero. Using the nonlocal kernel $J_\delta $ with $\delta =\sqrt{0.1},d=2$, 
we show its profile projection in 2D by figure \ref{f:1}.
\begin{figure}[h]
    \centering
    \includegraphics[scale=0.6]{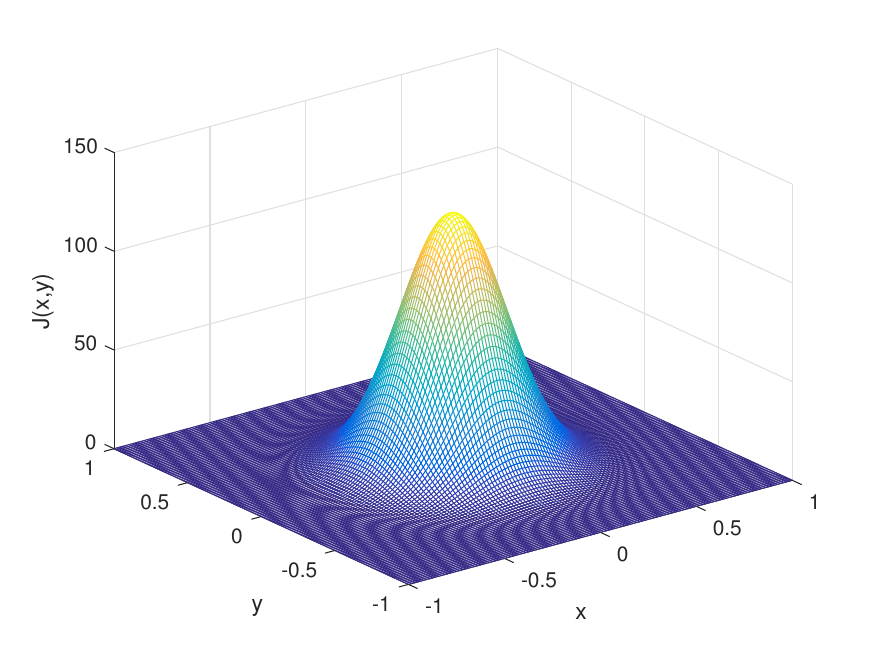}
    \caption{The profile of nonlocal kernel $J_\delta $ defined in \eqref{eqJ}}
    \label{f:1}
\end{figure}
\par
In the following numerical examples, 
we choose some numerical simulations in two dimension to verify the 
unconditional energy stability and temporal convergence rate.
We specify the rectangular domain as
$\varOmega =\left[ -1,1 \right] \times \left[ -1,1 \right]$, and 
the finite difference scheme is used in space, where $N=100$. 
Note that we denote the first-order ESI-SAV scheme \eqref{NCH1} by BDF1, 
the second-order ESI-SAV scheme based on the Crank-Nicolson 
formula \eqref{NCHCN} by CN, 
the high-order ESI-SAV schemes based on BDFk formula 
\eqref{NCH2}, \eqref{NCH3}, \eqref{NCH4}
by BDFk (k=2,3,4).
In order to verify the effectiveness of the fast solver (i.e. using FTT and 
FCG algorithm ), we compare them with the CPU time for the direct 
solver (i.e. solving with $\bar{\phi} = \bar{A}^{-1}\bar{b}$),
and abbreviate the fast solver and the direct solver as FS and DS respectively.

In our practical calculation, we will specify the operator 
$\mathcal{L}$, the bulk potential $F\left( \phi \right) $ as $\bar{\mathcal{L}}$, 
$\bar{F}( \phi )$ respectively:
\begin{align*}
    &\bar{\mathcal{L}}=\mathcal{L}+\frac{\beta}{\varepsilon ^2},\\
    &\bar{F}\left( \phi \right) =\frac{1}{4}\left( \phi ^2-1-\beta \right) ^2,
\end{align*}
where $\beta>0$ is a suitable parameter to ensure sufficient dissipation of the 
implicit part of the numerical scheme \cite{34}.
In the following numerical examples, when the error of the absolute 
value of the energy at two consecutive moments is less than the 
tolerance value of $1e-8$, we consider that the phase transition 
process has reached a steady state.\par

\textbf{Example 1}
For the first numerical example, we consider the NCH equation equipped 
with the Gaussian kernel $J_{\delta}$ and the required parameters are 
$M=1, \varepsilon^2$ =0.1, $\delta$=$\varepsilon$, $T$=0.05.  
In order to perform the refinement test of the time step, we use a linear 
refinement path in time:
$\varDelta t=T/T_n,T_n=2^n,n=3,4,\cdots ,10$, 
$\varDelta t$ is the time step.
One thing to notice is that we have no exact solution to compare, 
then we choose the difference between the results on continuous coarse 
grid and fine grid as the error calculation. 
That is, by taking the linear refinement path, 
we calculate the convergence rate by Cauchy error, 
e.g. $\lVert \phi ^{n+1}-\phi ^n \rVert _2$. 
Now we give the following initial condition:
\begin{align*}
    \phi _0\left( x,y \right) =0.8\sin \left( \pi x \right) \sin \left( \pi y \right).
\end{align*}
\par
The comparison between the fast solver (FS) and the direct solver (DS) is shown in
Table \ref{t:31} and Table \ref{t:32}, where the time is in seconds. 
It can be seen that the fast solver reduces the CPU time 
and greatly improves the computational efficiency.
Table \ref{t:3} shows the $L^2$-norm errors and 
temporal convergence rates of the schemes 
proposed in the previous section. 
As the time step decreases, the convergence rate gets closer to 
the optimal order of convergence.\par
\begin{table}[p]
    \centering    
    \renewcommand\arraystretch{1.5}  
    \caption{The CUP time(s) for direct solver (DS) and fast solver (FS) which are used for \textbf{Example 1} by the first-order BDF1 scheme}
    \label{t:31}
    \begin{tabular}{llllllll}
        \hline        
        $T_n$ & $2^3$ & $2^4$ & $2^5$ & $2^6$ & $2^7$ & $2^8$ & $2^9$\\ 
        \hline
        DS & 212.69 & 213.19 & 397.17 & 756.74 & 1430.69 & 2794.43 & 5547.35\\
        FS & 3.99 & 6.10 & 6.39 & 10.68 & 18.40 & 32.22 & 56.31\\ 
        \hline
    \end{tabular}  
\end{table}
\begin{table}[htp]
    \centering    
    \renewcommand\arraystretch{1.5}  
    \caption{The CUP time(s) for direct solver (DS) and fast solver(FS) 
    which are used for \textbf{Example 1} by the second-order CN
    scheme}
    \label{t:32}
    \begin{tabular}{llllllll}
        \hline        
        $T_n$ & $2^3$ & $2^4$ & $2^5$ & $2^6$ & $2^7$ & $2^8$ & $2^9$\\ 
        \hline
        DS & 132.44 & 230.01 & 421.47 & 779.01 & 1519.85 & 3001.38 & 5954.90\\
        FS & 2.88 & 4.05 & 6.58 & 11.13 & 18.55 & 32.10 & 57.36  \\ 
        \hline
    \end{tabular}  
\end{table}
\par
\begin{table}[hp]
    \centering    
    \renewcommand\arraystretch{1.5}  
    \caption{the $L^2$ errors and temporal convergence rates of the five schemes for the \textbf{Example 1}}
    \label{t:3}
    \begin{tabular}{llllllll}
        \hline        
        \multicolumn{2}{c}{Tn} & $2^3$ & $2^4$ & $2^5$ & $2^6$ & $2^7$ & $2^8$\\ \hline
        \multirow{2}{*}{BDF1}
        & Error & 1.16e-2 & 6.35-3 & 3.32e-3 & 1.70e-3 & 8.60e-4 & 4.32e-4\\
        & Rate & --- & 0.8686 & 0.9341 & 0.9672 & 0.9836 & 0.9918 \\ 
        \hline
        \multirow{2}{*}{CN}
        & Error & 1.79e-3 & 4.69e-4 & 1.21e-4 & 3.07e-5 & 7.76e-6 & 1.95e-6 \\
        & Rate & --- & 1.9326 & 1.9551 & 1.9748 & 1.9868 & 1.9932 \\ 
        \hline
        \multirow{2}{*}{BDF2}
        & Error & 1.85e-2 & 5.15e-3 & 1.35e-3 & 3.48e-4 & 8.80e-5 & 2.22e-5 \\
        & Rate & --- & 1.8453 & 1.9257 & 1.9626 & 1.9811 & 1.9906\\ 
        \hline
        \multirow{2}{*}{BDF3}
        & Error & 7.69e-3 & 1.08e-3 & 1.45e-4 & 1.88e-5 & 2.40e-6 & 3.03e-7 \\
        & Rate & --- & 2.8275 & 2.9030 & 2.9451 & 2.9704 & 2.9846 \\ 
        \hline
        \multirow{2}{*}{BDF4}
        & Error & 9.40e-4 & 7.71e-5 & 5.47e-6 & 3.62e-7 & 2.33e-8 & 1.47e-9 \\
        & Rate & --- & 3.6075 & 3.8174 & 3.9167 & 3.9608 & 3.9810 \\ 
        \hline
    \end{tabular}  
\end{table}\par
\textbf{Example 2}
For the second numerical example, we choose the benchmark problem \cite{32}
governed by the NCH equation equipped with the Gaussian kernel $J_{\delta}$,
and the initial condition is 
\begin{align*}
    \phi _0\left(x,y\right)=\sum_{i=1}^2{-\tan\text{h}\left(\frac{\sqrt{
        \left(x-a_i\right)^2+\left(y-b_i\right)^2}-R_0}{\sqrt{2}\varepsilon}\right)}.
\end{align*}
\par
In addition, we need to know some parameters, the radius $R_0 =0.36$,
the two bubble centers $\left( a_1,b_1 \right)=(0.4,0)$ and $\left(a_2,b_2\right)=(-0.4,0)$.
We adopt the third order ESI-SAV BDF3 scheme for this example.
The required parameters are $\varepsilon =0.02, \delta =\varepsilon$, $\varDelta t=1e-3$.
\par
The evolution of the phase field variable $\phi$ at $t=0,0.1,0.6,1,5,10$ are 
shown in Figure \ref{figure2.1}. 
From the snapshots, we observe that the two bubbles are just close 
to each other initially. With the time evolution, the two bubbles fuse with each other, which can 
also be said that the small bubble is absorbed by the large bubble. Moreover, 
due to the mass conservation of phase field variables in the NCH model, 
the two small bubbles will turn into a large round bubble.
After $t=11$, they tend to a stable phase state, that is, there is only one large bubble. 
The result is consistent with that in \cite{32}. 
In Figure \ref{figure2.2}, we draw the evolution of the energy curve, 
which verifies that the modified discrete energy does conform to the law of energy dissipation law.\par
\begin{figure}[htp]
    \centering
    \subfloat[t=0]{\includegraphics[width=0.32\textwidth,height=0.32\textwidth]{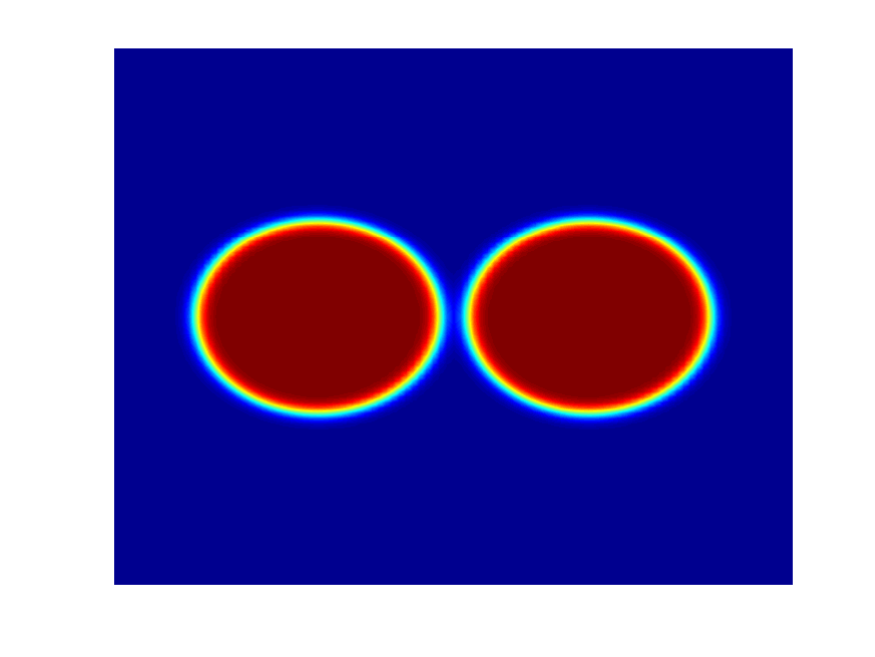}
    \label{a}}
    \hfill
    \subfloat[t=0.1]{\includegraphics[width=0.32\textwidth,height=0.32\textwidth]{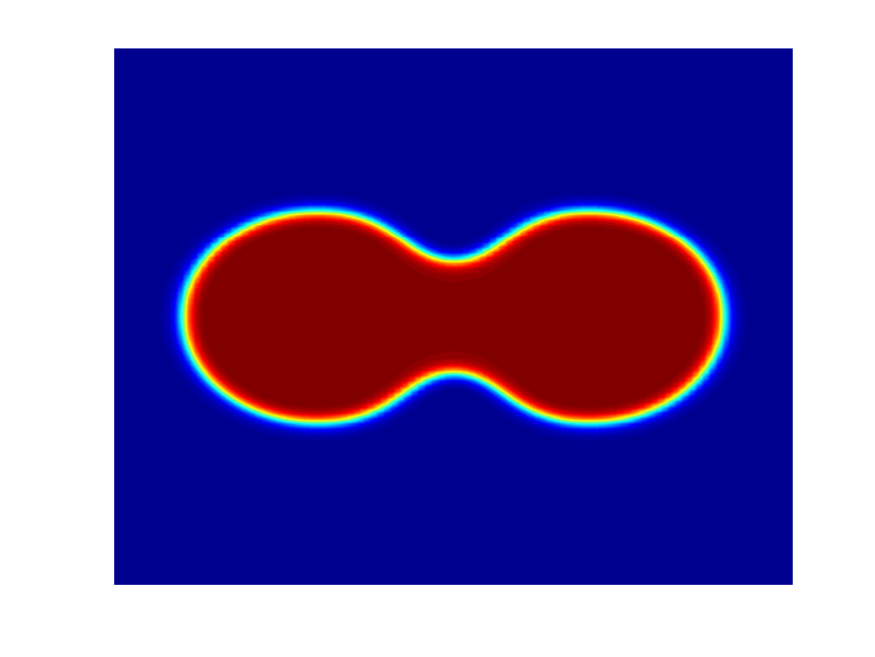}
    \label{b}}
    \hfill
    \subfloat[t=0.6]{\includegraphics[width=0.32\textwidth,height=0.32\textwidth]{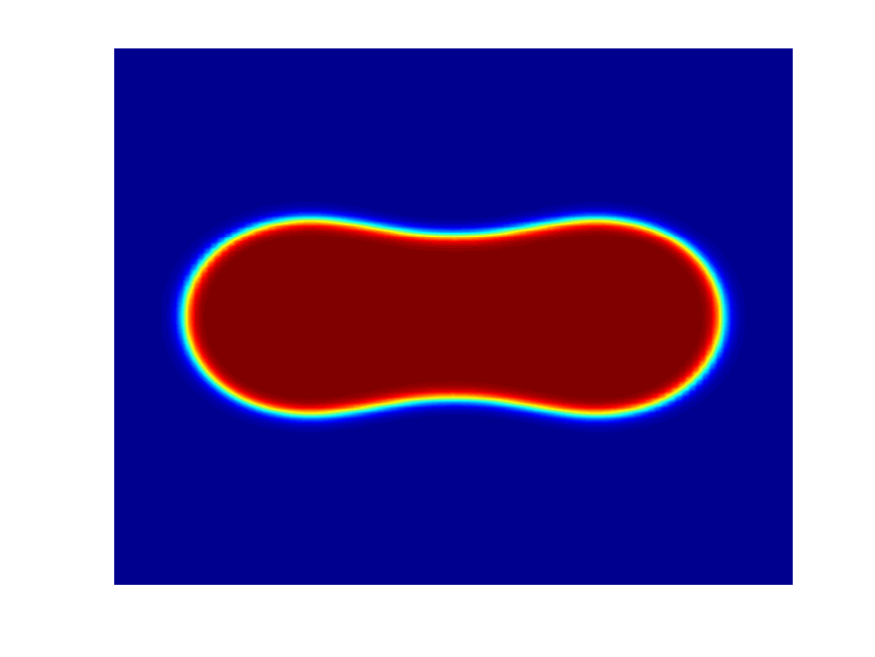}
    \label{c}}
    \hfill
    \\
    \subfloat[t=1]{\includegraphics[width=0.32\textwidth,height=0.32\textwidth]{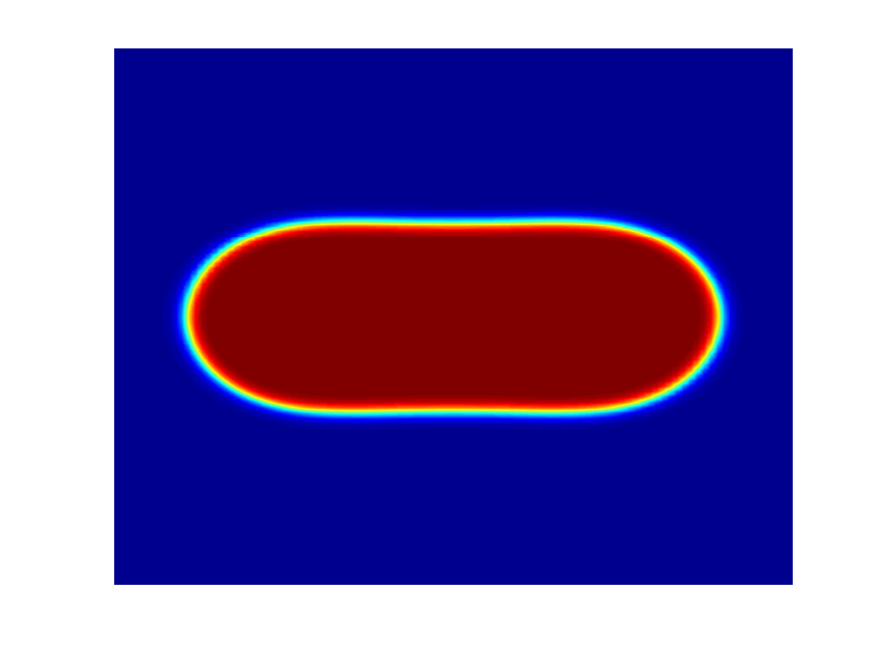}
    \label{d}}
    \hfill
    \subfloat[t=5]{\includegraphics[width=0.32\textwidth,height=0.32\textwidth]{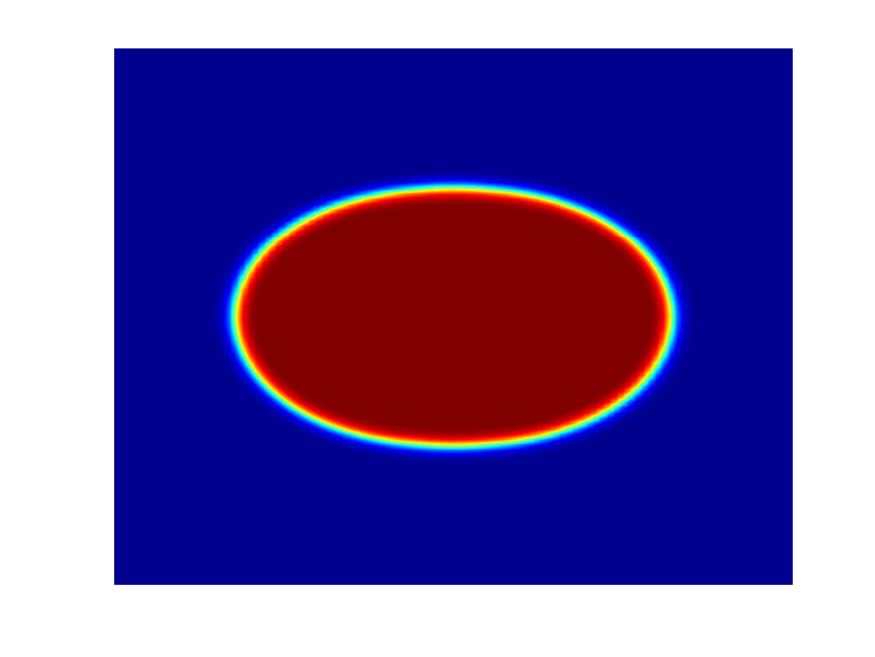}
    \label{e}}
    \hfill
    \subfloat[t=10]{\includegraphics[width=0.32\textwidth,height=0.32\textwidth]{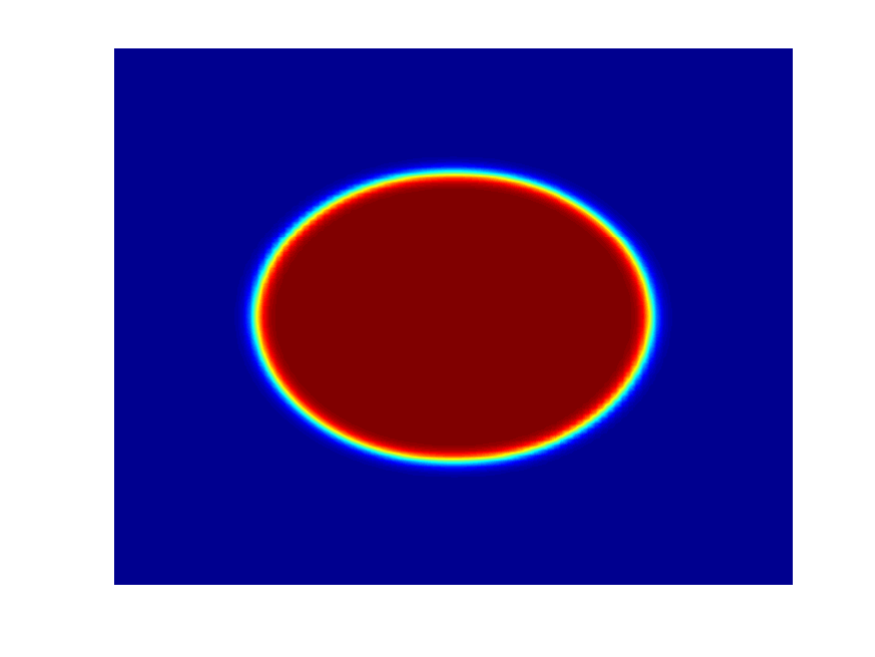}
    \label{f}}
    \hfill
    \caption{Snapshots of the phase variable $\phi$ are taken at $t=0,0.1,0.6,1,5,10$ for \textbf{Example 2}}
    \label{figure2.1}
    \hfill
\end{figure}
\begin{figure}
    \centering
    \includegraphics[scale=0.6]{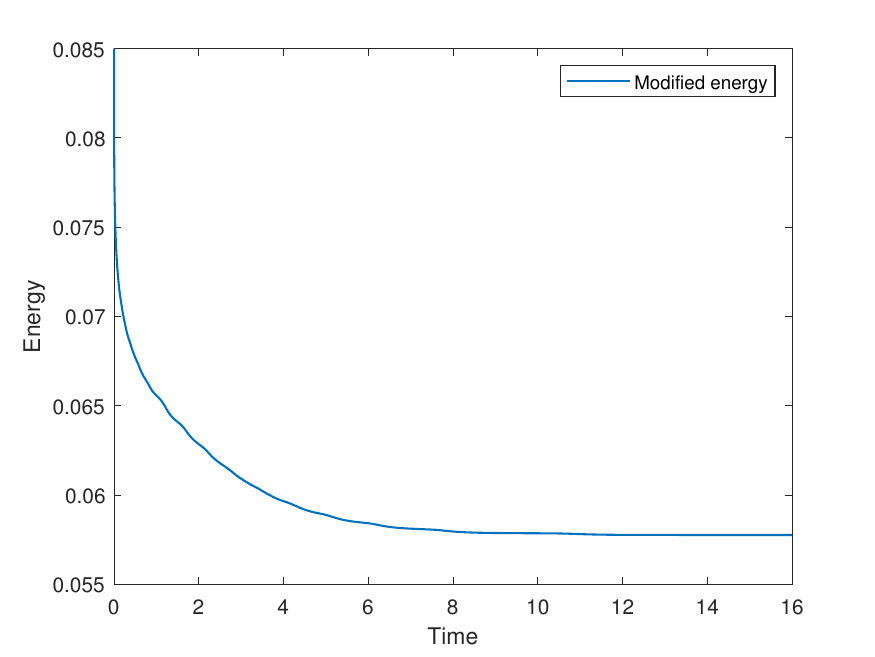}
    \caption{Time evolution of the modified discrete free energy functional \eqref{EE} for \textbf{Example 2}}
    \label{figure2.2}
\end{figure}\par
\textbf{Example 3}
In this example, we study the coarsening dynamic model driven by the NCH equation.
We still use the same 
kernel function specified in \eqref{eqJ} to simulate the long time behavior of 
the phase transition. We set the initial conditions as the randomly perturbed 
concentration field as shown below:
\begin{equation*}
    \phi _0\left( x,y \right) =\phi_a+\phi_b rand\left( x,y \right),
\end{equation*}
where $rand\left( x,y \right)$ is the random number in $[-1,1]$ with zero mean.
\par
The other required parameters are shown in Table \ref{t:43}.
\begin{table}[hp]
    \centering 
    \caption{The required parameters in \textbf{Example 3}}
    \label{t:43}   
    \begin{tabular}{cccc}
        \toprule
        $\varepsilon$ & $\delta$ & $\varDelta t$ \\
        \midrule
        $0.02$ & $0.02$ & $1e-3$  \\
        \bottomrule  
    \end{tabular}  
\end{table}
\par
The long-time coarsening dynamical behaviors of the phase separation are shown 
in Figure \ref{figure3.1} and Figure \ref{figure3.2}.
The snapshots are taken at $t=0,0.2,1,8,15$, and the time of final steady state.
We adjust the values of $\phi_a$ and $\phi_b$ to obtain various coarsening dynamics model.
In Figure \ref{figure3.1}, we perform numerical simulations for initial values 
with $\phi_a=0,\phi_b=0.1$,
and we observe that the final equilibrium solution after $t=39$ shows 
a red circle.
In Figure \ref{figure3.2}, we choose $\phi_a=0.1$ and $\phi_b=0.001$, 
the final equilibrium solution after $t=33$ shows 
a blue circle, but the value of $\phi$ in the circle becomes the exact opposite of that in Figure \ref{figure3.1}.
The results is consistent with that in \cite{31,32}.
In Figure \ref{figure3.3}, we can see that the original discrete energy decreases over time, 
which satisfies the energy dissipation law.
\begin{figure}[htp]
    \centering
    \subfloat[t=0]{\includegraphics[width=0.32\textwidth,height=0.32\textwidth]{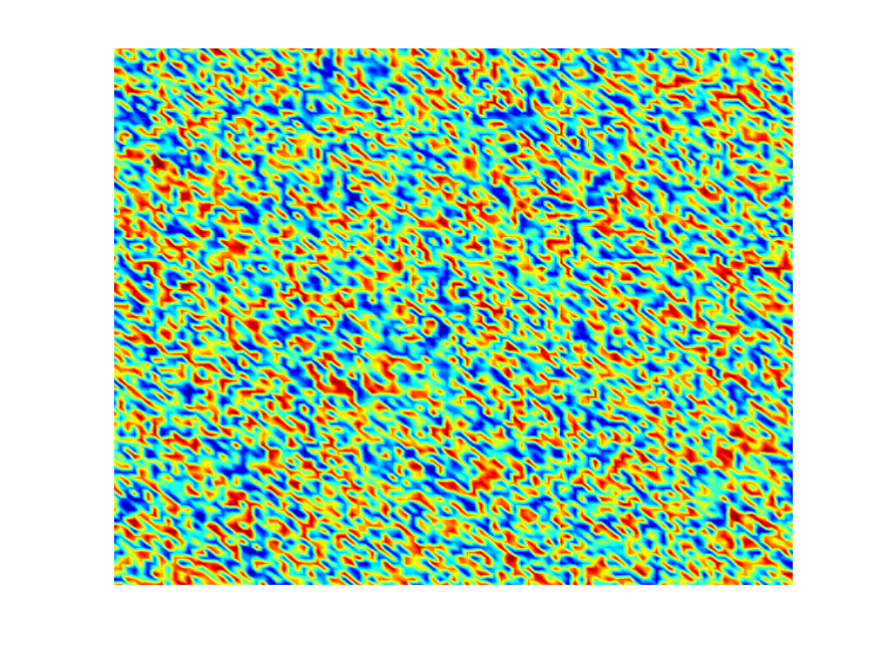}
    \label{a}}
    \hfill
    \subfloat[t=0.2]{\includegraphics[width=0.32\textwidth,height=0.32\textwidth]{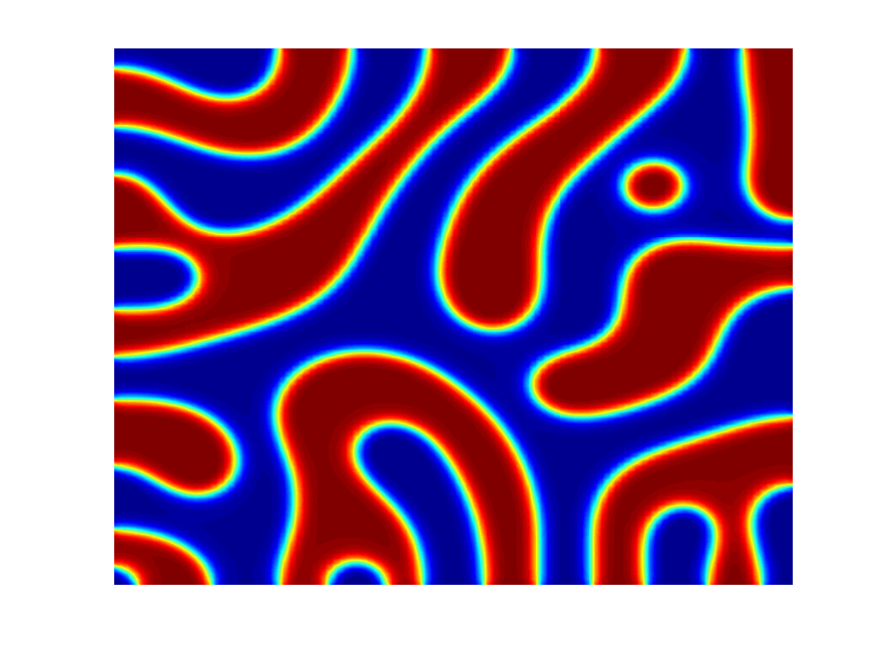}
    \label{b}}
    \hfill
    \subfloat[t=1]{\includegraphics[width=0.32\textwidth,height=0.32\textwidth]{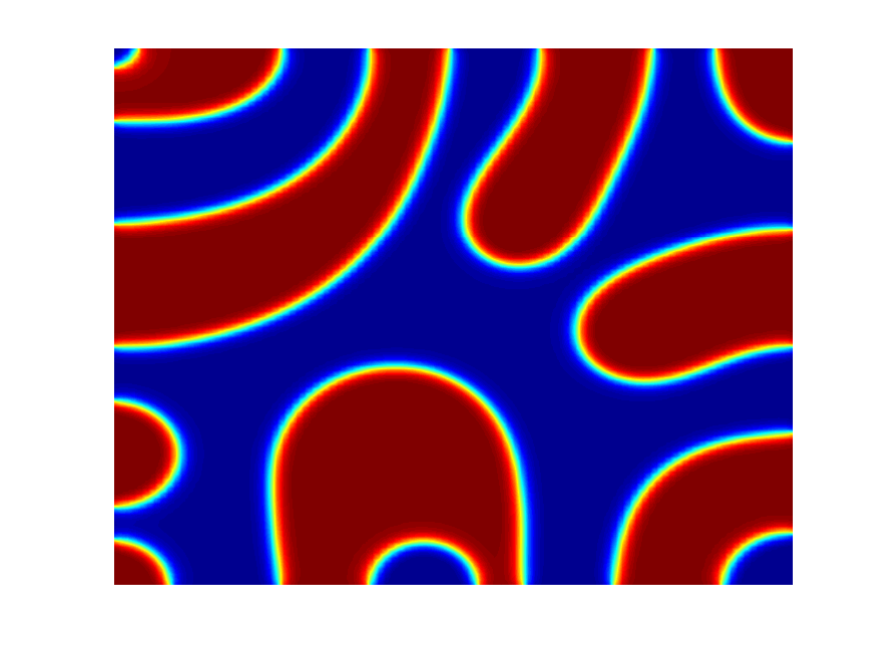}
    \label{c}}
    \hfill
    \\
    \subfloat[t=8]{\includegraphics[width=0.32\textwidth,height=0.32\textwidth]{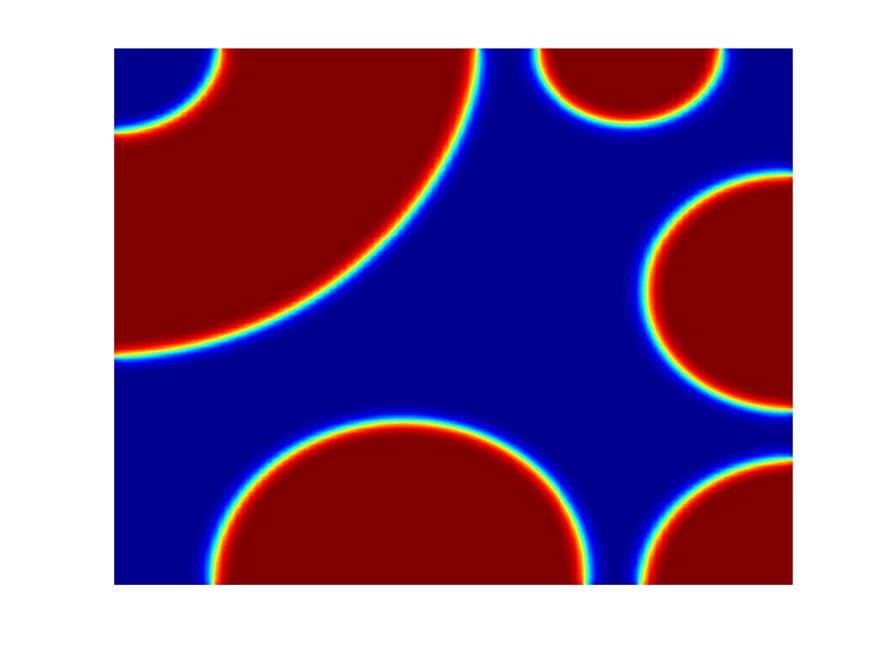}
    \label{d}}
    \hfill
    \subfloat[t=15]{\includegraphics[width=0.32\textwidth,height=0.32\textwidth]{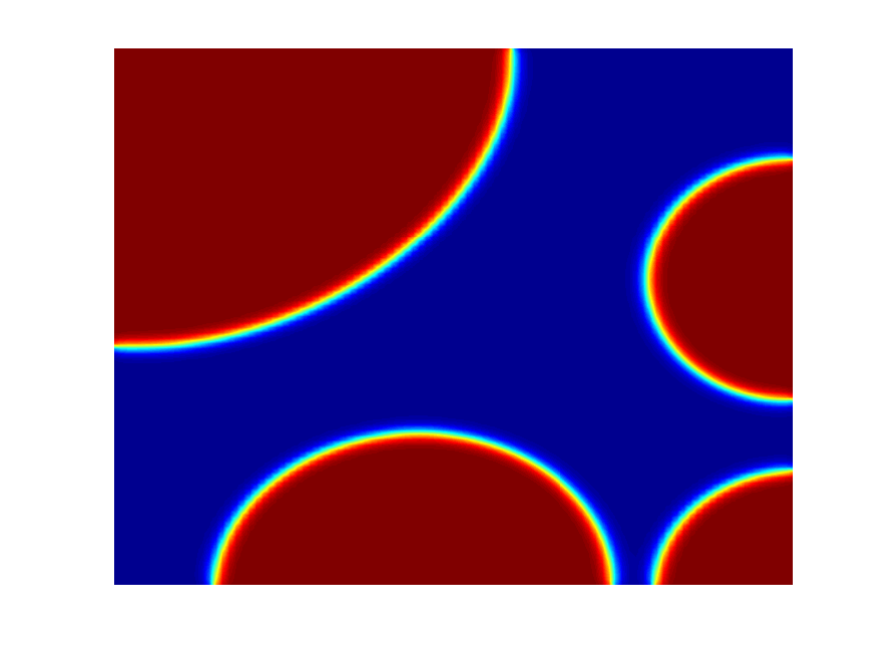}
    \label{f}}
    \hfill
    \subfloat[t=39]{\includegraphics[width=0.32\textwidth,height=0.32\textwidth]{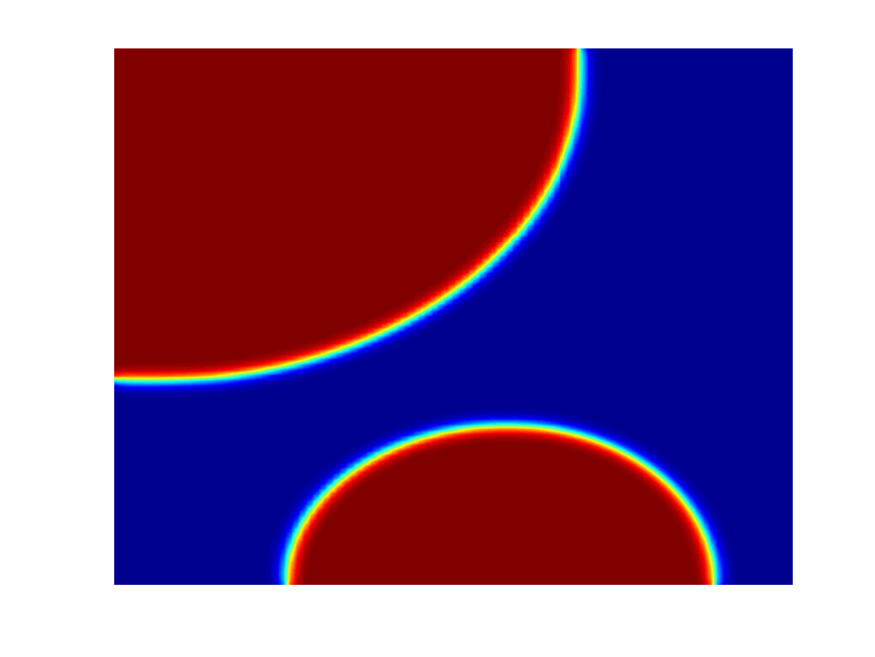}
    \label{e}}
    \hfill
    \caption{Snapshots of the phase variable $\phi$ at $t=0,0.2,1,8,15,39$ with $\phi_a=0,\phi_b=0.1$ for \textbf{Example 3}}
    \label{figure3.1}
    \hfill
\end{figure}
\begin{figure}[htp]
    \centering
    \subfloat[t=0]{\includegraphics[width=0.32\textwidth,height=0.32\textwidth]{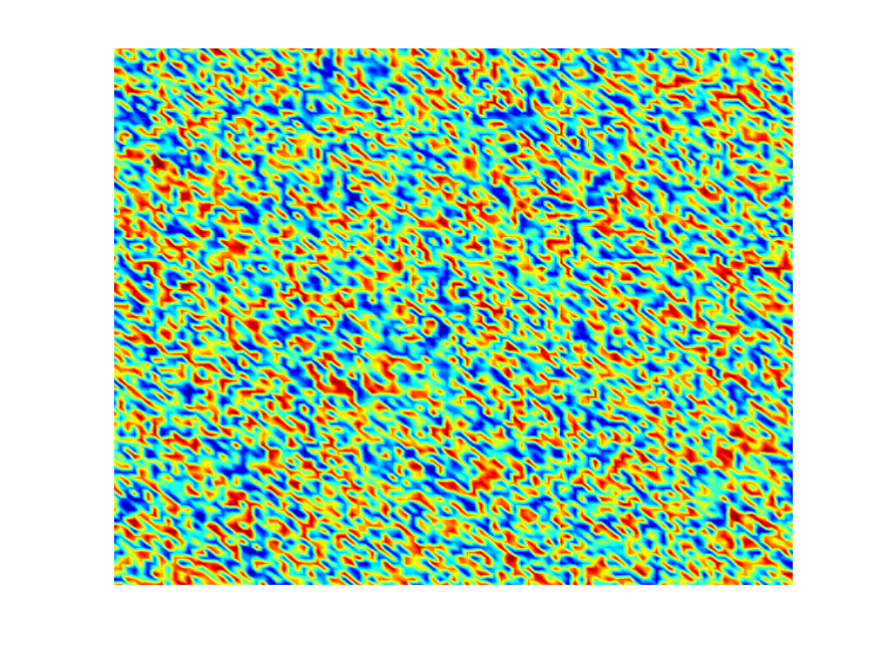}
    \label{a}}
    \hfill
    \subfloat[t=0.2]{\includegraphics[width=0.32\textwidth,height=0.32\textwidth]{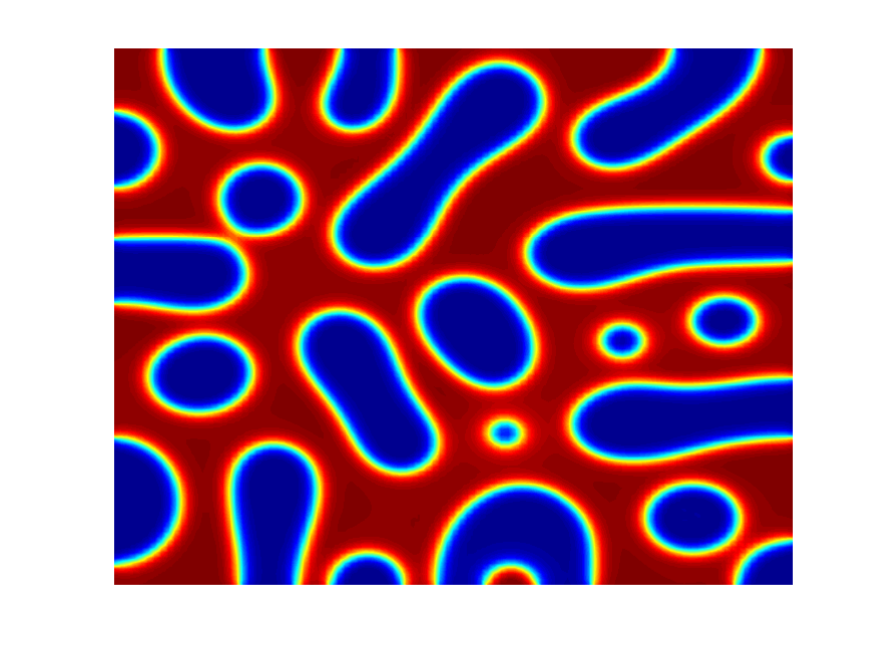}
    \label{b}}
    \hfill
    \subfloat[t=1]{\includegraphics[width=0.32\textwidth,height=0.32\textwidth]{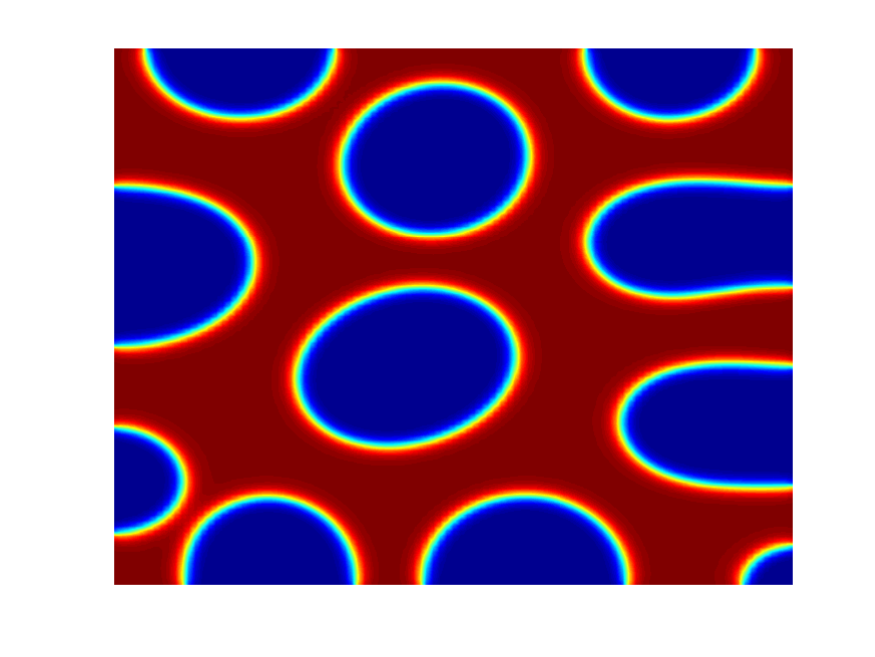}
    \label{c}}
    \hfill
    \\
    \subfloat[t=8]{\includegraphics[width=0.32\textwidth,height=0.32\textwidth]{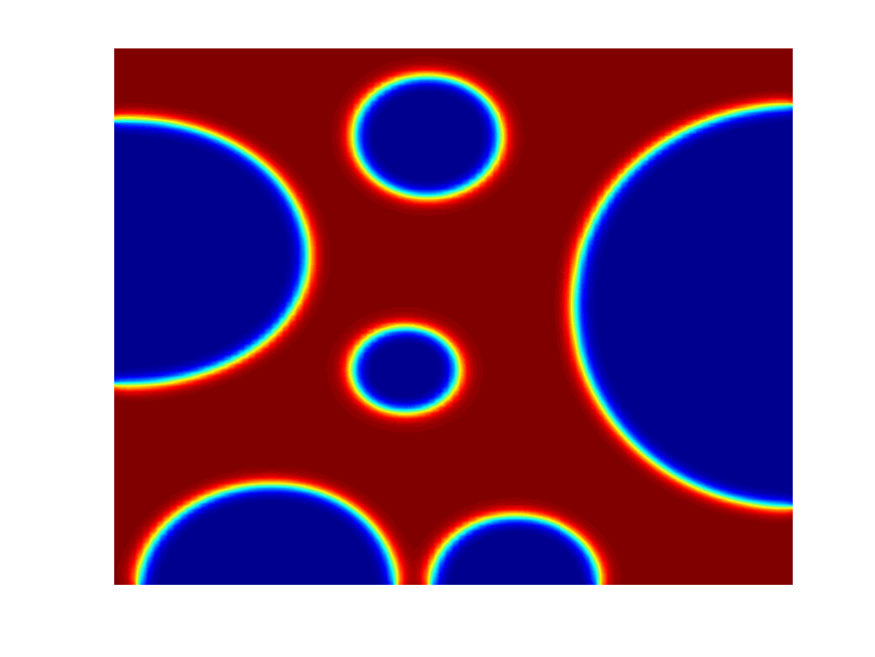}
    \label{d}}
    \hfill
    \subfloat[t=15]{\includegraphics[width=0.32\textwidth,height=0.32\textwidth]{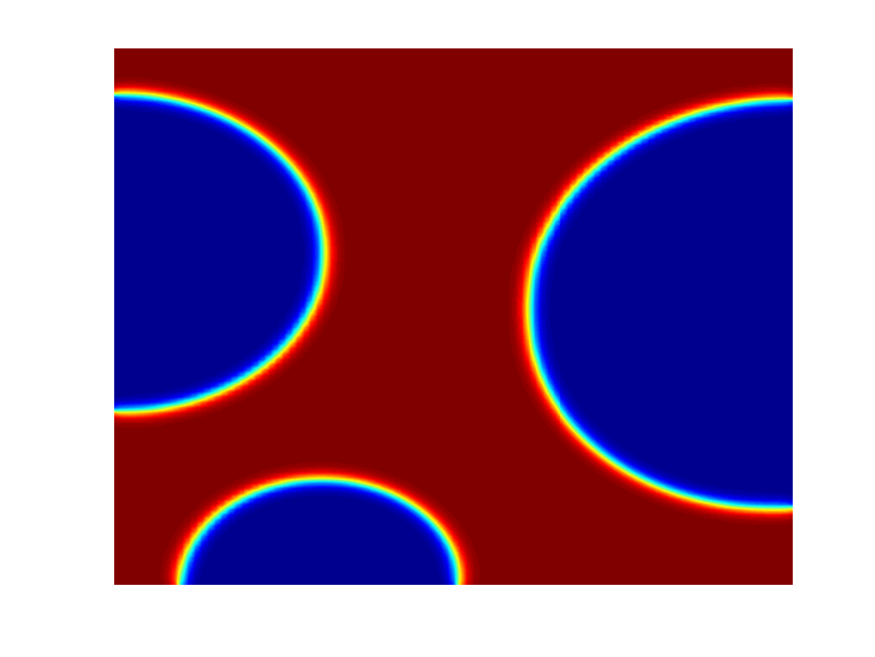}
    \label{f}}
    \hfill
    \subfloat[t=33]{\includegraphics[width=0.32\textwidth,height=0.32\textwidth]{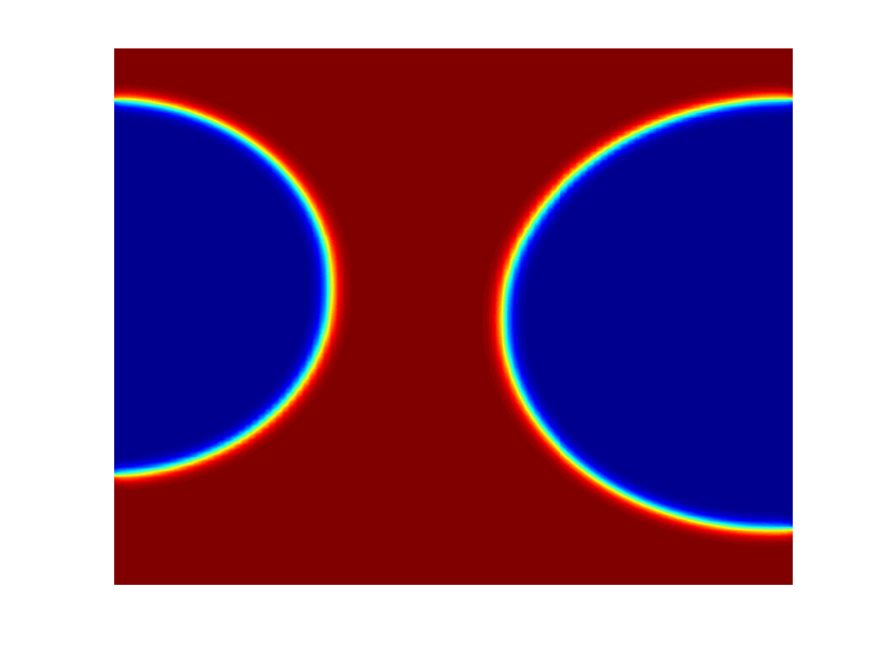}
    \label{e}}
    \hfill
    \caption{Snapshots of the phase variable $\phi$ at $t=0,0.2,1,8,15,33$ with $\phi_a=0.1,\phi_b=0.001$ for \textbf{Example 3}}
    \label{figure3.2}
    \hfill
\end{figure}
\begin{figure}[htp]
    \centering
    \subfloat[$\phi_a=0,\phi_b=0.1$]{\includegraphics[scale=0.45]{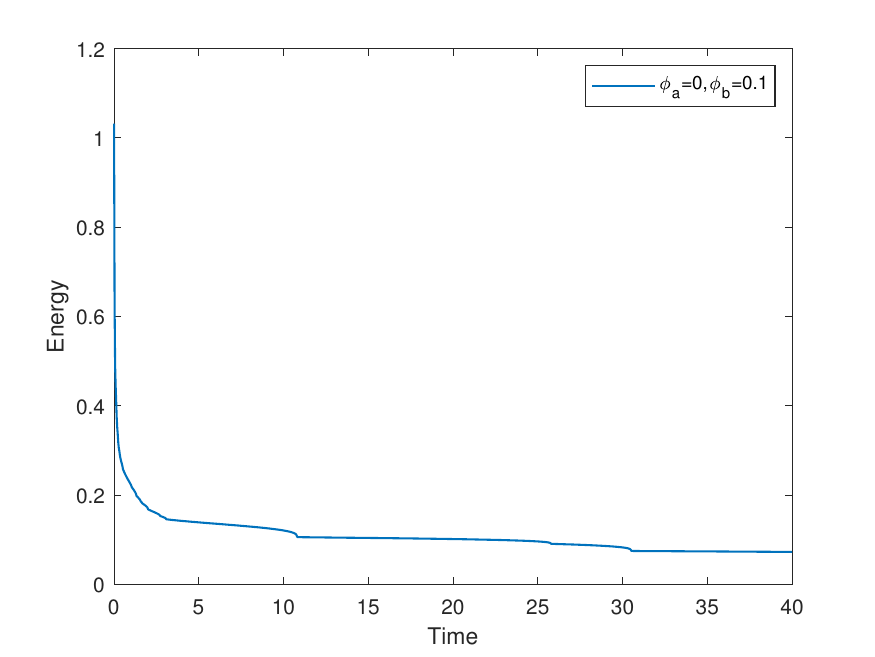}
    \label{a}}
    \hfill
    \subfloat[$\phi_a=0.1,\phi_b=0.001$]{\includegraphics[scale=0.45]{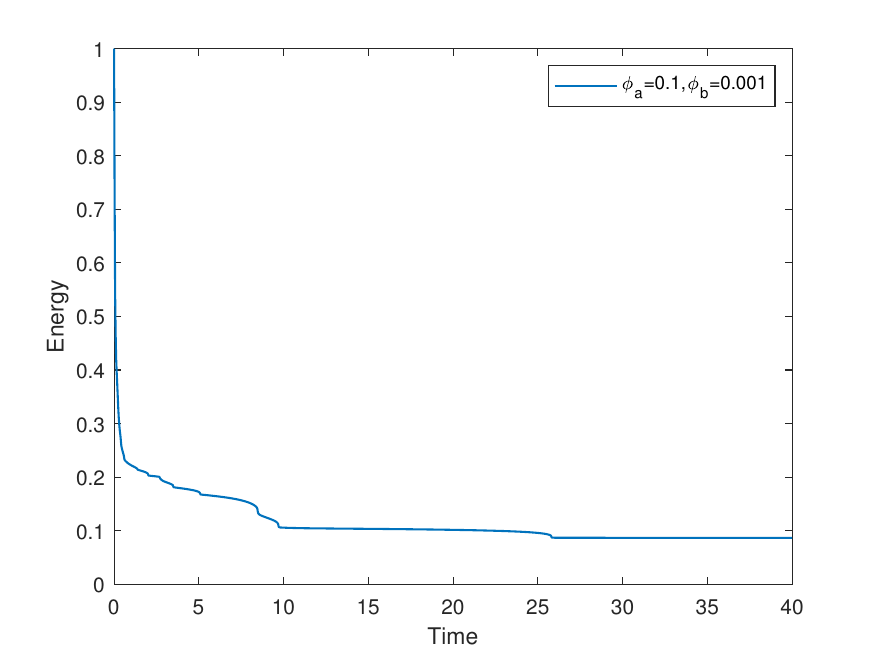}
    \label{b}}
    \hfill
    \caption{Time evolution of the original discrete free energy functional \eqref{EEE} for the coarsening dynamic model
    with 
    (a) $\phi_a=0, \phi_b=0.1$; (b) $\phi_a=0.1, \phi_b=0.001$.}
    \label{figure3.3}
\end{figure}
\section{Conclusion}\par
In this paper, we construct the high-order and effective linear schemes 
based on the ESI-SAV for solving the NCH equation 
with general nonlinear potential. 
Through the comparison of the LCH and NCH models, 
we analyze the nonlocal model and related properties. 
We focus on the method of numerical discretization
with high-order accuracy (in time) of 
the NCH model and use the ESI-SAV method to deal with 
nonlocal diffusion term implicitly and nonlinear term explicitly 
by introducing a scalar auxiliary variable. 
Under the assumption of positive kernel, we prove the unconditional 
energy stability of the semi-discrete scheme carefully and rigorously. 
Furthermore, the discrete energy dissipation law of the 
proposed numerical schemes is guaranteed. In addition, 
we consider integrable kernel only to make the calculation simple. 
Due to the properties of Guassian kernel parameterized by $\delta$ of 
nonlocal operator and the special format of coefficient matrix, 
we can use a fast solver based on FFT and FCG to reduce the amount 
of storage and computation cost required in fully discrete format,
which does greatly shorten the CPU time. Finally, 
we demonstrate the effectiveness, 
accuracy and unconditional energy stability of the proposed numerical scheme
by using different classical numerical experiments.
\section*{Acknowledgement}\par
This work was supported in part by the National Natural Science Foundation of China under Grants 11971272, 12001336.

\bibliographystyle{plain}
\bibliography{NCH}

\begin{thebibliography}{10}

\bibitem{61}
Georgios Akrivis, Buyang Li, and Dongfang li.
\newblock Energy-decaying extrapolated rk--sav methods for the allen--cahn and
  cahn--hilliard equations.
\newblock {\em SIAM Journal on Scientific Computing}, 41:A3703--A3727, 2019.

\bibitem{21}
Andrew~J. Archer and Markus Rauscher.
\newblock {Dynamical density functional theory for interacting Brownian
  particles: stochastic or deterministic?}
\newblock {\em Journal of Physics A General Physics}, 37:9325--9333, 2004.

\bibitem{9}
Nicola~J. Armstrong, Kevin~J. Painter, and Jonathan~A. Sherratt.
\newblock {A continuum approach to modelling cell-cell adhesion.}
\newblock {\em Journal of theoretical biology}, 243:98--113, 2006.

\bibitem{10}
Nicola~J. Armstrong, Kevin~J. Painter, and Jonathan~A. Sherratt.
\newblock {Adding Adhesion to a Chemical Signaling Model for Somite Formation}.
\newblock {\em Bulletin of Mathematical Biology}, 71:1--24, 2009.

\bibitem{37}
Peter~W. Bates.
\newblock {On some nonlocal evolution equations arising in materials science.
  Nonlinear dynamics and evolution equations}.
\newblock {\em Fields Institute Communications}, 48:13--52, 2006.

\bibitem{25}
Peter~W. Bates and Jianlong Han.
\newblock {The Dirichlet boundary problem for a nonlocal Cahn-Hilliard
  equation}.
\newblock {\em Journal of Mathematical Analysis and Applications},
  311:289--312, 2005.

\bibitem{26}
Peter~W. Bates and Jianlong Han.
\newblock {The Neumann boundary problem for a nonlocal Cahn-Hilliard equation}.
\newblock {\em Journal of Differential Equations}, 212:235--277, 2005.

\bibitem{1}
John~W. Cahn and John~E. Hilliard.
\newblock {Free Energy of a Nonuniform System. I. Interfacial Free Energy}.
\newblock {\em Journal of Chemical Physics}, 28:258--267, 1958.

\bibitem{11}
Arnaud Chauviere, Haralambos Hatzikirou, Ioannis~G. Kevrekidis, John~S.
  Lowengrub, and Vittorio Cristini.
\newblock {Dynamic density functional theory of solid tumor growth: Preliminary
  models.}
\newblock {\em AIP advances}, 2:8032--234, 2012.

\bibitem{40}
Chuanjun Chen and Xiaofeng Yang.
\newblock {Highly efficient and unconditionally energy stable semi-discrete
  time-marching numerical scheme for the two-phase incompressible flow
  phase-field system with variable-density and viscosity}.
\newblock {\em Science China Mathematics}, 65:1--26, 2022.

\bibitem{2}
Longqing Chen and Yunzhi Wang.
\newblock {The continuum field approach to modeling microstructural evolution}.
\newblock {\em JOM}, 48:13--18, 1996.

\bibitem{18}
Pierluigi Colli, Sergio Frigeri, and Maurizio Grasselli.
\newblock {Global existence of weak solutions to a nonlocal
  Cahn-Hilliard-Navier-Stokes system}.
\newblock {\em Journal of Mathematical Analysis and Applications},
  386:428--444, 2011.

\bibitem{28}
Qiang Du, Max~D. Gunzburger, Richard~B. Lehoucq, and Kun Zhou.
\newblock {Analysis and Approximation of Nonlocal Diffusion Problems with
  Volume Constraints}.
\newblock {\em SIAM Review}, 54:667--696, 2012.

\bibitem{34}
Qiang Du, Lili Ju, Xiao Li, and Zhonghua Qiao.
\newblock Stabilized linear semi-implicit schemes for the nonlocal
  cahn-hilliard equation.
\newblock {\em Journal of Computational Physics}, 363:39--54, 2018.

\bibitem{17}
Robert Evans.
\newblock {The nature of the liquid-vapour interface and other topics in the
  statistical mechanics of non-uniform, classical fluids}.
\newblock {\em Advances in Physics}, 28:143--200, 1979.

\bibitem{24}
Herbert Gajewski and Klaus G{\"a}rtner.
\newblock {On a nonlocal model of image segmentation}.
\newblock {\em Zeitschrift f{\"u}r angewandte Mathematik und Physik ZAMP},
  56:572--591, 2005.

\bibitem{19}
Giambattista Giacomin and Joel~L Lebowitz.
\newblock {Phase segregation dynamics in particle systems with long range
  interactions. I. Macroscopic limits}.
\newblock {\em Journal of Statistical Physics}, 87:37--61, 1997.

\bibitem{20}
Giambattista Giacomin and Joel~L Lebowitz.
\newblock {Phase Segregation Dynamics in Particle Systems with Long Range
  Interactions II: Interface Motion}.
\newblock {\em SIAM Journal on Applied Mathematics}, 58:1707--1729, 1998.

\bibitem{62}
Yuezheng Gong, Jia Zhao, and Qi~Wang.
\newblock Arbitrarily high-order unconditionally energy stable sav schemes for
  gradient flow models.
\newblock {\em Computer Physics Communications}, 249:107033, 2020.

\bibitem{30}
Zhen Guan, John~S. Lowengrub, Cheng Wang, and Steven~M. Wise.
\newblock {Second order convex splitting schemes for periodic nonlocal
  Cahn-Hilliard and Allen-Cahn equations}.
\newblock {\em Journal of Computational Physics}, 277:48--71, 2014.

\bibitem{29}
Zhen Guan, Cheng Wang, and Steven~M. Wise.
\newblock {A convergent convex splitting scheme for the periodic nonlocal
  Cahn-Hilliard equation}.
\newblock {\em Numerische Mathematik}, 128:377--406, 2014.

\bibitem{41}
Dianming Hou, Mejdi Aza{\"i}ez, and Chuanju Xu.
\newblock {A variant of scalar auxiliary variable approaches for gradient
  flows}.
\newblock {\em Journal of Computational Physics}, 395:307--332, 2019.

\bibitem{64}
Fukeng Huang and Jie Shen.
\newblock A new class of implicit-explicit bdf $k$ sav schemes for general
  dissipative systems and their error analysis.
\newblock {\em Computer Methods in Applied Mechanics and Engineering},
  392:114718, 2022.

\bibitem{7}
Maryna Kapustina, Denis Tsygankov, Jia Zhao, Timothy Wessler, Xiaofeng Yang,
  Alex Chen, Nathan~P Roach, Timothy~C. Elston, Qi~Wang, Ken Jacobson, and
  M.~Gregory Forest.
\newblock {Modeling the Excess Cell Surface Stored in a Complex Morphology of
  Bleb-Like Protrusions}.
\newblock {\em PLoS Computational Biology}, 12:e1004841, 2016.

\bibitem{14}
Christos~N. Likos, Bianca~M. Mladek, Dieter Gottwald, and Gerhard Kahl.
\newblock {Why do ultrasoft repulsive particles cluster and crystallize?
  Analytical results from density-functional theory.}
\newblock {\em The Journal of chemical physics}, 126:224502, 2007.

\bibitem{6}
Chun Liu and Jie Shen.
\newblock {A phase field model for the mixture of two incompressible fluids and
  its approximation by a Fourier-spectral method}.
\newblock {\em Physica D: Nonlinear Phenomena}, 179:211--228, 2003.

\bibitem{52}
Huan Liu, Aijie Cheng, Hong Wang, and Jia Zhao.
\newblock {Time-fractional Allen-Cahn and Cahn-Hilliard phase-field models and
  their numerical investigation}.
\newblock {\em Computers \& Mathematics with Applications}, 76:1876--1892,
  2018.

\bibitem{42}
Zhengguang Liu and Xiaoli Li.
\newblock {The Exponential Scalar Auxiliary Variable (E-SAV) Approach for Phase
  Field Models and Its Explicit Computing}.
\newblock {\em SIAM Journal on Scientific Computing}, 42:B630--B655, 2020.

\bibitem{32}
Zhengguang Liu and Xiaoli Li.
\newblock {The fast scalar auxiliary variable approach with unconditional
  energy stability for nonlocal Cahn-Hilliard equation}.
\newblock {\em Numerical Methods for Partial Differential Equations},
  37:244--261, 2020.

\bibitem{33}
Zhengguang Liu and Xiaoli Li.
\newblock {A highly efficient and accurate exponential semi-implicit scalar
  auxiliary variable (ESI-SAV) approach for dissipative system}.
\newblock {\em Journal of Computational Physics}, 447:110703, 2021.

\bibitem{3}
John~S. Lowengrub, Andreas R{\"a}tz, and Axel Voigt.
\newblock {Phase-field modeling of the dynamics of multicomponent vesicles:
  Spinodal decomposition, coarsening, budding, and fission.}
\newblock {\em Physical review. E, Statistical, nonlinear, and soft matter
  physics}, 79:031926, 2009.

\bibitem{15}
Umberto Marini~Bettolo Marconi and Pedro Tarazona.
\newblock {Dynamic density functional theory of fluids}.
\newblock {\em The Journal of Chemical Physics}, 110:8032--8044, 1999.

\bibitem{23}
R.~C. Merton.
\newblock {Option pricing when underlying stock returns are discontinuous}.
\newblock {\em Journal of Financial Economics}, 3:125--144, 1976.

\bibitem{5}
C.~Miehe, Martina Hofacker, and Fabian Welschinger.
\newblock {A phase field model for rate-independent crack propagation: Robust
  algorithmic implementation based on operator splits}.
\newblock {\em Computer Methods in Applied Mechanics and Engineering},
  199:2765--2778, 2010.

\bibitem{38}
Len~M Pismen.
\newblock {Nonlocal diffuse interface theory of thin films and the moving
  contact line}.
\newblock {\em Physical Review E}, 64:021603, 2001.

\bibitem{12}
Robert~C. Rogers.
\newblock {A nonlocal model for the exchange energy in ferromagnetic
  materials}.
\newblock {\em Journal of Integral Equations and Applications}, 3:85--127,
  1991.

\bibitem{13}
Robert~C. Rogers.
\newblock {Some remarks on nonlocal interactions and hysteresis in phase
  transitions}.
\newblock {\em Continuum Mechanics and Thermodynamics}, 8:65--73, 1994.

\bibitem{35}
Jie Shen, Jie Xu, and Jiang Yang.
\newblock {The scalar auxiliary variable (SAV) approach for gradient flows}.
\newblock {\em Journal of Computational Physics}, 353:407--416, 2018.

\bibitem{51}
Jie Shen and Xiaofeng Yang.
\newblock {Numerical approximations of Allen-Cahn and Cahn-Hilliard equations}.
\newblock {\em Discrete and Continuous Dynamical Systems}, 28:1669--1691, 2010.

\bibitem{63}
Zengqiang Tan and Huazhong Tang.
\newblock A general class of linear unconditionally energy stable schemes for
  the gradient flows.
\newblock {\em Journal of Computational Physics}, 464:111372, 2022.

\bibitem{46}
Zhifeng Weng, Shuying Zhai, and Xinlong Feng.
\newblock {Analysis of the operator splitting scheme for the Cahn-Hilliard
  equation with a viscosity term}.
\newblock {\em Numerical Methods for Partial Differential Equations},
  35:1949--1970, 2019.

\bibitem{50}
Xufeng Xiao, Xinlong Feng, and Jinyun Yuan.
\newblock {The stabilized semi-implicit finite element method for the surface
  Allen-Cahn equation}.
\newblock {\em Discrete and Continuous Dynamical Systems-series B},
  22:2857--2877, 2017.

\bibitem{31}
Xiaofeng Yang and Jia Zhao.
\newblock {Efficient linear schemes for the nonlocal Cahn-Hilliard equation of
  phase field models}.
\newblock {\em Computer Physics Communications}, 235:234--245, 2018.

\end{thebibliography}
\end{document}